\documentclass{amsart}

\DeclareRobustCommand{\SkipTocEntry}[5]{}

\usepackage[utf8]{inputenc}
\usepackage{amsmath}
\usepackage{mathtools}
\usepackage{amssymb}
\usepackage{mismath}
\usepackage{amsthm}

\usepackage{dsfont}

\usepackage[margin=1.05in]{geometry}

\usepackage{array}
\usepackage{longtable}

\usepackage{setspace}

\usepackage{hyperref}

\usepackage{enumerate}

\usepackage{mathrsfs}

\usepackage{dsfont}

\usepackage[OT2,T1]{fontenc}

\newtheorem{theorem}{Theorem}[section]
\newtheorem{corollary}[theorem]{Corollary}
\newtheorem{lemma}[theorem]{Lemma}
\newtheorem{proposition}[theorem]{Proposition}
\newtheorem{conjecture}[theorem]{Conjecture}

\newtheorem*{parity}{Parity Conjecture}

\newtheorem*{paritytwists}{Parity Conjecture for Twists}

\theoremstyle{definition}
\newtheorem{definition}[theorem]{Definition}

\newtheorem{example}[theorem]{Example}

\newtheorem{remark}[theorem]{Remark}
\newtheorem{notation}[theorem]{Notation}

\usepackage{tikz}
\usetikzlibrary{shapes,positioning}
\usetikzlibrary{arrows,chains,matrix,positioning,scopes}

\makeatletter
\tikzset{join/.code=\tikzset{after node path={%
			\ifx\tikzchainprevious\pgfutil@empty\else(\tikzchainprevious)%
			edge[every join]#1(\tikzchaincurrent)\fi}}}
\makeatother

\tikzset{>=stealth',every on chain/.append style={join},
	every join/.style={->}}

\DeclareSymbolFont{cyrletters}{OT2}{wncyr}{m}{n}
\DeclareMathSymbol{\Sha}{\mathalpha}{cyrletters}{"58}

\mathtoolsset{showonlyrefs}

\title{Root numbers and parity phenomena}

\author{Lilybelle Cowland Kellock}
\address{University College London, London WC1H 0AY, UK}
\email{lilybelle.kellock.20@ucl.ac.uk}

\author{Vladimir Dokchitser}
\address{University College London, London WC1H 0AY, UK}
\email{v.dokchitser@ucl.ac.uk}

\subjclass[2020]{Primary: 11g05, Secondary: 11g10, 11g40}

\begin{document}
	
	\begin{abstract}
		The parity conjecture has a long and distinguished history. It gives a way of predicting the existence of points of infinite order on elliptic curves without having to construct them, and is responsible for a wide range of unexplained arithmetic phenomena. It is one of the main consequences of the Birch and Swinnerton-Dyer conjecture and lets one calculate the parity of the rank of an elliptic curve using root numbers. In this handbook, we explain how to use local root numbers of elliptic curves to realise some of these phenomena, with an emphasis on explicit calculations. The text is aimed at a `user' and, as such, we will not be concerned with the proofs of known cases of the parity conjecture, but instead we will demonstrate the use of the theory by means of examples.
	\end{abstract}

	\maketitle
	
	\tableofcontents
	
	\section{Introduction}
	
	\subsection{The parity conjecture}
	
	The Mordell--Weil rank of an elliptic curve $E$ over a number field $K$ remains mysterious. Calculating it is generally either hard or impossible, and there is currently no method that will provably work for every curve. All we have is ways of finding upper bounds by calculating the rank of Selmer groups, lower bounds by searching for $K$-rational points, and fingers to cross that these coincide. In general, it can be hard to find out if an elliptic curve has even one $K$-rational point of infinite order.
	
	The parity conjecture predicts points of infinite order and most known phenomena about elliptic curves effortlessly, after a simple calculation. It quantifies how the behaviour of $E$ over completions of $K$ should control the parity of the rank of $E/K$. It does this using the \textit{global root number of $E/K$}, $w(E/K)\in\{\pm 1\}$, which is defined as the product of easily computable local factors at each place $v$ of $K$, \textit{local root numbers}. They are straightforward to compute since, despite having a non-constructive definition, they have been classified for all places (see \S \ref{rootnossection}). 
	
	\begin{parity}
		Let $E$ be an elliptic curve over a number field $K$. Then 
		\begin{equation}
			(-1)^{\emph{\text{rk}}(E/K)}=w(E/K),
		\end{equation}
		where $w(E/K)$ is the global root number of $E/K$. 
	\end{parity}
	
	Note that if $w(E/K)=-1$, the parity conjecture implies that the rank is odd and, in particular, greater than $0$. The parity conjecture is often the only way of predicting that a point of infinite order exists on an arbitrary elliptic curve. For now, take for granted that if $E$ is a semistable elliptic curve defined over $K$, then 
	\begin{equation}
		w(E/K)=(-1)^{m+u},
	\end{equation}
	where $m$ is the number of primes of $K$ where $E$ has split multiplicative reduction and $u$ is the number of infinite places of $K$ (see Corollary \ref{semistable}).

	\begin{example}\label{999}
		Let us take
		\begin{equation}
			E:y^2 - 23y = x^3 - 99997x^2 - 17x + 42. 
		\end{equation}
		Magma \cite{magma} tells us that the rank of $E/\mathbb{Q}$ is $0$, so that $E$ does not have a rational point of infinite order. However, it also returns \texttt{false}, indicating that it has not proved that $0$ is the rank of the curve\footnote{Of course, we could have done more extensive calculations; Magma only searches for points which have height below $15$ and performs only $2$- and $4$-descent by default.}. With a simple calculation, the parity conjecture tells us that this curve should have a point of infinite order. Indeed, we have $\Delta_E=17\cdot 655943686625481101$ and non-split multiplicative reduction at both primes. Hence, since $\mathbb{Q}$ has one infinite place, the above result tells us that $w(E/\mathbb{Q})=(-1)^1=-1$ and so the parity conjecture implies that the rank of $E/\mathbb{Q}$ is odd and that $E$ has a point of infinite order.
	\end{example}

	One could use the conjecture of Birch and Swinnerton-Dyer to calculate Mordell--Weil ranks of elliptic curves. However, utilising the Birch--Swinnerton-Dyer conjecture requires computing $L$-functions, and, not only is this a rather lengthy calculation, we cannot provably compute $\text{ord}_{s=1}L(E/\mathbb{Q},s)$ unless it is $0$, $1$, $2$ or $3$, nor do we know the analytic continuation of $L$-functions of elliptic curves over general number fields. The beauty of the parity conjecture is that it is free of the conjectural machinery of $L$-functions and is an entirely arithmetical statement, yet it still allows us to predict the existence of points of infinite order. 
	
	In terms of applications, parity-related arguments have many notable consequences. For example, a known case of the $p$-parity conjecture\footnote{Namely that $\dim_{\mathbb{F}_p}\textup{Sel}_p(E/\mathbb{Q})-\dim_{\mathbb{F}_p}E(\mathbb{Q})[p]$ is even if and only if $w(E/\mathbb{Q})=+1$.} was used by Bhargava and Shankar to prove that a positive proportion of elliptic curves over $\mathbb{Q}$ satisfy the Birch--Swinnerton-Dyer conjecture \cite{bhargava}. The parity conjecture settles most cases of the congruent number problem (see \S \ref{congnumberprob}) and has been used to prove many results about the ranks of families of elliptic curves \cite{desjardins,helfgott,rizzo,rohrlichfamilies}. A variant on the $2$-parity conjecture was used to prove that Hilbert's Tenth Problem has a negative answer over rings of integers of number fields \cite{hilbert10} and results on the parity of Selmer ranks were used to prove that certain Kummer surfaces satisfy the Hasse principle \cite{skoro}. Similar arguments have also been used to prove results on class numbers of e.g. simplest cubic fields \cite{washington}. Root numbers also seem to have played a role in the early investigation of Heegner points, see e.g. \cite{gross} Chapter 5, Corollary 21.3 and Conjecture 23.2.

	The parity conjecture is easily deduced from the Birch--Swinnerton-Dyer conjecture together with the Hasse--Weil conjecture. The latter states that the completed $L$-function has analytic continuation to the whole of the complex plane and satisfies the functional equation $\Hat{L}(E/K,s)=w(E/K)\Hat{L}(E/K,2-s)$, where $w(E/K)$ is the global root number. Assuming the Hasse-Weil conjecture, $w(E/K)=+1$ if and only if the completed $L$-function is symmetric around $s=1$, which happens if and only if its order of vanishing at $s=1$ is even. This tells us that
	\begin{equation}
		w(E/K)=(-1)^{\text{ord}_{s=1}\Hat{L}(E/K,s)},
	\end{equation}
	which, assuming the Birch--Swinnerton-Dyer conjecture, gives the parity conjecture.

	\subsection{A historical note}\label{selmerstuff}
	The investigation of parity-type questions appears to have begun with calculations of the parity of ranks of Selmer groups. In his 1954 paper `A conjecture concerning rational points on cubic curves' \cite{selmer}, Selmer gave an explicit definition of the first and second descent for $E/\mathbb{Q}$ with a rational $2$-isogeny by writing out the equations they lead to. Selmer wrote that he had performed numerical calculations that suggest ``when the first descent indicated at most three generators, then none or two of these seem to be excluded by the second descent'' (at most three as this was the case for all curves he numerically tested). Based on this, he conjectured that
	
	\begin{center}
		1. \textit{``The second descent excludes an even number of generators''}.
	\end{center}
	
	In his 1962 paper `Arithmetic on curves of genus 1, III, The Tate–Šafarevič
	and Selmer groups' \cite{cassels2}, where he coined the term `Selmer group', Cassels extended Selmer's conjecture and proved that the dimension of the image of $\text{Sel}_{p^n}(E/K)$ under the map induced by multiplication by $p^{n-1}$ differs from that of $\text{Sel}_p(E/K)$, considered as a vector space over $\mathbb{F}_p$, by an even integer. This also treats the case of third descent, fourth descent, and so on. In \cite{selmer}, Selmer also conjectured the stronger statement, that
	
	\begin{center}
		2. \textit{``The number of generators indicated by a first descent differs from the true number of generators by an even number''}.  
	\end{center}
	
	Extending Selmer's second conjecture to general $p$, as Cassels did for Selmer's first conjecture, it says that the rank of the $p$-Selmer group determines the parity of the Mordell--Weil rank and is equivalent to saying that $\delta_p$ is even, where $\delta_p$ is the multiplicity of $\mathbb{Q}_p/\mathbb{Z}_p$ in $\Sha(E/K)[p^\infty]$ and is conjecturally $0$. In light of the parity conjecture and Selmer's second conjecture, we would expect that 
	\begin{equation}
		(-1)^{\text{rk}_p(E/K)}=w(E/K),
	\end{equation}
	where $\text{rk}_p(E/K)=\text{rk}(E/K)+\delta_p$ is the $\mathbb{Z}_p$-corank of the $p$-infinity Selmer group $\lim_{n\to \infty} \text{Sel}_{p^n}(E/K)$. This is known as the \textit{$p$-parity conjecture}. Although we will not discuss it or its applications in this handbook, tackling the $p$-parity conjecture has proved to be more fruitful than tackling the parity conjecture itself and there have been numerous works on the subject (see \S \ref{using} for references). In this context, one can prove parity-related statements without assuming the finiteness of the Tate--Shafarevich group (the fabled Shafarevich--Tate conjecture), while the parity conjecture itself appears to be, as of now, unapproachable without this assumption.

	\subsection{Using this handbook}\label{using}
	
	This exposition is designed to explain how to calculate local and global root numbers of elliptic curves, and give a taste of what sorts of phenomena are predicted by the parity conjecture. The latter are concerned with the existence of rational points on elliptic curves for which there is as yet no construction, and that remain open problems without assuming the parity conjecture. As of now, the only known systematic construction of points of infinite order are Heegner points and self-points \cite{self1,self2}, both using the modular parameterisation. We hope that our examples might inspire some new ideas for the construction of rational points. 
	
	In \S \ref{rootnossection}, we describe how to calculate local and global root numbers of elliptic curves, along with root numbers of quadratic twists of elliptic curves and twists of elliptic curves by self-dual Artin representations. We also give a first introduction as to how to use these to predict the existence of points of infinite order on elliptic curves. In \S \ref{examples}, we illustrate how the parity conjecture predicts many unexplained arithmetic phenomena about rational points on elliptic curves. For a quick overview of these, see the table of contents. In \S \ref{minconjtwists}, we propose an analogue of the minimalist conjecture for Artin twists of elliptic curves over $\mathbb{Q}$, and discuss its consequences. The conjecture gives the expected Galois module structure of $E(F)\otimes_\mathbb{Z}\mathbb{C}$, for a fixed Galois extension $F/\mathbb{Q}$ and a generic elliptic curve $E/\mathbb{Q}$.
	
	The parity conjecture is known assuming the finiteness of the Tate--Shafarevich group \cite{bsdmodsq,monsky}, and it is also known that these approaches to controlling the parity of the rank do not extend to controlling the Mordell--Weil rank modulo $n$ for $n>2$ \cite{modn}. We will not be concerned with approaches to the proof of the parity conjecture; for a self-contained exposition on the proof of the parity conjecture assuming the finiteness of the Tate--Shafarevich group (and of the proof of $p$-parity conjecture for elliptic curves over $\mathbb{Q}$) see \cite{timparity}. We will confine ourselves to the setting of elliptic curves, although the parity conjecture is formulated for abelian varieties. For some recent results on these conjectures see \cite{cesna,coates,pparity,finsha,regconst,bsdmodsq,abelianvar,green,kim,adam,nekovar1,nekovar2,nekovar3,nekovar4}; for results on root numbers of abelian varieties see \cite{bisatt,sabitova}.

	\subsection{Notation}
	We will use the following notation.
	\begin{align}
		& E && \text{ an elliptic curve}; \\
		& K && \text{ a number field}; \\
		& \Delta_K && \text{ the discriminant of $K$} ; \\
		& v && \text{ a place of $K$}; \\
		& K_{v} && \text{ the completion of $K$ at a place $v$}; \\
		& \langle\cdot,\cdot\rangle && \text{ the usual inner product of (characters of) representations}; \\
		& \Delta_E && \text{ the discriminant of a given Weierstrass equation for $E$}; \\
		& c_4,c_6 && \text{ the usual invariants of an elliptic curve (see \cite{silverman} p.46)}; \\
		& N_E && \text{ the conductor of $E$}; \\
		& E_d && \text{ the quadratic twist of $E$ by } d\in {K^\times}/{K^ {\times 2}}; \\
		& w(E/\mathcal{K}) && \text{ the local root number of $E$ over a local field $\mathcal{K}$ (see \S \ref{localrootnos})}; \\
		& w(E/K) && \text{ the global root number of $E/K$ (see \S \ref{globalrootnos})}; \\
		& w(E/K,\rho) && \text{ the global root number of $E/K$ twisted by an Artin representation $\rho$ (see \S \ref{paritytwists})}.
	\end{align}
	Most curves will be labelled with their Cremona label, if they have one. 
	
	\addtocontents{toc}{\SkipTocEntry}
	\section*{Acknowledgements}
	
The authors would like to thank James Bell, Kestutis Česnavičius, John Cremona, Julie Desjardins, Tim Dokchitser, Zev Klagsbrun, Chao Li, Rachel Newton, Ariel Pacetti, Ross Paterson, Ottavio Rizzo, David Rohrlich, Mohammad Sadek and Alexei Skorobogatov for useful discussions and comments. The second author is very grateful to the late John Coates for his interest, encouragement and discussions on the topic of explicit arithmetic phenomena of elliptic curves.

The first author was supported by the Engineering and Physical Sciences Research Council [EP/L015234/1], the EPSRC Centre for Doctoral Training in Geometry and Number Theory (The London School of Geometry and Number Theory) and University College London. The second author was supported by a Royal Society University Research Fellowship.
	
		\section{Root numbers}\label{rootnossection}
	
	\subsection{The global root number}\label{globalrootnos}
	
	\begin{definition}
		Let $E$ be an elliptic curve over a number field $K$. The \textit{global root number} $w(E/K)\in \{\pm 1\}$ is defined as the product of \textit{local root numbers} $w(E/K_v)\in\{\pm 1\}$,
		\begin{equation}
			w(E/K)=\prod_{v} w(E/K_v),
		\end{equation}
		where the product runs over all places $v$ of $K$, including infinite ones.
	\end{definition}
	
	Local root numbers of elliptic curves are defined using epsilon-factors of Weil--Deligne representations (see \cite{deligne} for the original definitions and \cite{timparity} Section 3.3 for a more down-to-earth introduction). For a finite place $v$, the local root number is computed by looking at the action of the absolute Galois group of $K_v$ on the $\ell$-adic Tate module. For the purpose of this exposition, we will not concern ourselves with formal definitions, but will focus on how to calculate local root numbers in practice.
	
	\begin{remark}
		For an elliptic curve defined over $\mathbb{Q}$, the local root number at a prime $p$ agrees with the eigenvalue of the associated Atkin-Lehner involution for the associated modular form. This follows from the corresponding statement for modular forms (see \cite{schmidt} Theorem 3.2.2, for example) together with the local Langlands conjecture for $\text{GL}_2$ and the modularity of elliptic curves over $\mathbb{Q}$.
	\end{remark}

	\subsection{How to calculate local root numbers}\label{localrootnos}
	Local root numbers of elliptic curves $E/K$ have been classified for all places $v$ of $K$, so, for most purposes, one can avoid the technical definitions. 
	
	\begin{theorem}\label{rootnos}
		Let $E$ be an elliptic curve over a local field $\mathcal{K}$ of characteristic zero. When $\mathcal{K}$ is non-Archimedean, let $k$ be its residue field and let $v:\mathcal{K}^\times\twoheadrightarrow \mathbb{Z}$ denote the normalised valuation with respect to~$\mathcal{K}$. Let $\left(\frac{*}{k}\right)$ denote the quadratic residue symbol on $k^\times$ and $(a,b)_\mathcal{K}$ denote the Hilbert symbol in $\mathcal{K}$. 
		\begin{enumerate}[(i)]
			\item If $\mathcal{K}$ is Archimedean then $w(E/\mathcal{K})=-1$;
			\item If $E/\mathcal{K}$ has good reduction then $w(E/\mathcal{K})=+1$;
			\item If $E/\mathcal{K}$ has split multiplicative reduction then $w(E/\mathcal{K})=-1$;
			\item If $E/\mathcal{K}$ has non-split multiplicative reduction then $w(E/\mathcal{K})=+1$;
			\item If $E/\mathcal{K}$ has additive, potentially multiplicative reduction and $\textup{char}(k)\geq 3$ then $w(E/\mathcal{K})=\big(\frac{-1}{k}\big)$. 
			\item If $E/\mathcal{K}$ has additive, potentially multiplicative reduction and $\textup{char}(k)=2$ then $w(E/\mathcal{K})=(-1,-c_6)_\mathcal{K}$. In particular, if $\mathcal{K}=\mathbb{Q}_2$ then
			\begin{equation}
				w(E/\mathbb{Q}_2)=\begin{cases}
					-1 &\text{ if } c_6'\equiv 1\pmod{4}; \\
					+1 &\text{ if } c_6'\equiv 3\pmod{4},
				\end{cases}
			\end{equation}
			where $c_6'=c_6/2^{v(c_6)}$.
			\item If $E/\mathcal{K}$ has additive, potentially good reduction and $\textup{char}(k)\geq 5$ then $w(E/\mathcal{K})=(-1)^{\big\lfloor \frac{v(\Delta) \mid k\mid}{12}\big\rfloor}$, where $\Delta$ is the minimal discriminant of $E$. 
			\item If $E/\mathcal{K}$ has additive, potentially good reduction and $\textup{char}(k)\geq 3$, let $y^2=x^3+ax^2+bx+c$ be a Weierstrass equation for $E$ and let $\Delta_E$ be the discriminant of this model.
			\begin{itemize}
				\item If the Kodaira type\footnote{Recall that the Kodaira type of $E/\mathcal{K}$ can be calculated using Tate's algorithm (see \cite{adsilverman}, Ch. IV \S 9).} is $\rm{I}_0^*$, then
				\begin{equation}
					w(E/\mathcal{K})=\Big(\frac{-1}{k}\Big);
				\end{equation}
				\item If the Kodaira type is $\rm{III}$ or $\rm{III}^*$, then
				\begin{equation}
					w(E/\mathcal{K})=\Big(\frac{-2}{k}\Big);
				\end{equation}
				\item If the Kodaira type is $\rm{II}$, $\rm{IV}$, $\rm{II}^*$ or $\rm{IV}^*$, then there exists a Weierstrass equation for which $3\nmid v(c)$. For such an equation, we have
				\begin{equation}
					w(E/\mathcal{K})=\delta\cdot (\Delta_E,c)_\mathcal{K}\cdot\Big(\frac{v(c)}{k}\Big)^{v(\Delta_E)}\cdot\Big(\frac{-1}{k}\Big)^{\frac{v(\Delta_E)(v(\Delta_E)-1)}{2}}, 
				\end{equation}
				where $\delta\in\{\pm 1\}$ and $\delta=1$ if and only if $\sqrt{\Delta_E}\in \mathcal{K}$. 
			\end{itemize}
		\end{enumerate}
	\end{theorem}
	The semistable case follows from the definitions in \cite{deligne}, Rohrlich proved $(ii)$--$(vii)$ in \cite{rohrlich2} \S19 and \cite{rohrlich} Theorem 2 (the case of $\mathcal{K}=\mathbb{Q}_2$ in $(vi)$ can be found in \cite{connell} \S3), and $(viii)$ is a result of Kobayashi \cite{kobayashi} Theorem 1.1. The statements of $(i)$--$(v)$ and $(vii)$, as they are written, are found in \cite{bsdmodsq} Theorem 3.1. For elliptic curves over local fields of positive characteristic, there is a similar classification of root numbers in \cite{conrad} Theorem 3.1.
	
	Theorem \ref{rootnos} does not tell us how to calculate local root numbers when $E$ has additive potentially good reduction over a local field of residue characteristic $2$. For $E/\mathbb{Q}_2$, one can use a classification of Halberstadt \cite{halberstadt}, which was extended to non-minimal Weierstrass models and potentially multiplicative reduction by Rizzo (see \cite{rizzo} Table III). The tables that classify the local root number $w(E/\mathbb{Q}_2)$ can be found in Appendix~\ref{rootnosat2table}. Generalising Rizzo's and Halberstadt's tables to extensions of $\mathbb{Q}_2$ seems unfeasible as there would be many cases to check for large extensions. However, in the case of potentially good reduction over an extension of $\mathbb{Q}_2$, the local root number can be described in terms of $\text{Gal}(\mathcal{K}(E[3])/\mathcal{K})\subset \text{GL}(2,\mathbb{F}_3)$, see \cite{rootnos2}. There is also a formula (\cite{finsha} Theorem 1.12) for local root numbers in terms of the Tamagawa number of $E$ and the curves that are $2$-isogenous to $E$ over subfields of $\mathcal{K}(E[2])$; the Tamagawa numbers can be read off from Tate's algorithm \cite{tatealgorithm} (or see \cite{adsilverman} p. 365).
	
	\begin{corollary}\label{qweierstrass}
		Let $E$ be an elliptic curve over a local field $\mathcal{K}$ of characteristic zero and residue characteristic $p\geq 5$. Let $k$ denote the residue field of $\mathcal{K}$ and let $\left(\frac{*}{k}\right)$ denote the quadratic residue symbol on $k$. Then 
		\begin{equation}
			w(E/\mathbb{Q}_p) = \begin{cases}
				1 & \text{ if $E/\mathcal{K}$ has Kodaira type $\rm{I}_0$}; \\
				\big(\frac{-1}{k}\big) & \text{ if $E/\mathcal{K}$ has type $\rm{II}$, $\rm{II}^*$, ${\rm{I}}_n^*$ or ${\rm{I}}_0^*$}; \\
				\big(\frac{-2}{k}\big) & \text{ if $E/\mathcal{K}$ has type $\rm{III}$ or $\rm{III}^*$};\\
				\big(\frac{-3}{k}\big) & \text{ if $E/\mathcal{K}$ has type $\rm{IV}$ or $\rm{IV}^*$}; \\
				-\big(\frac{6b}{k}\big) & \text{ if $E/\mathcal{K}$ has type ${\rm{I}}_n$}, 
			\end{cases}
		\end{equation}
		where $E$ is given by a Weierstrass equation $y^2=x^3+ax+b$, with $a,b\in\mathcal{O}_\mathcal{K}$.
	\end{corollary}
	
	\begin{proof}
		For $\mathcal{K}=\mathbb{Q}_p$ this was proved in \cite{rohrlichfamilies} Proposition 2. If $E/\mathcal{K}$ has Kodaira type $\textup{I}_0$, $\textup{I}_0^*$, $\textup{I}_n^*$, $\textup{III}$ or $\rm{III}^*$, these expressions follow immediately from the theorem. If $E/\mathcal{K}$ has type $\textup{I}_n$, the root number depends on whether the reduction type is split or non-split. We may assume the Weierstrass equation for $E$ is minimal, since a change of model changes the value of $b$ by a $6$-th power. Write
		\begin{equation}
			\Bar{E}:y^2=(x-\alpha)^2(x-\beta) \pmod{m_{\mathcal{K}}}
		\end{equation}
		for the reduction of $E$ modulo the maximal ideal of $\mathcal{K}$. $E/\mathcal{K}$ has split multiplicative reduction if and only if $\alpha-\beta\in k^{\times 2}$. Since $-\alpha^2\beta=b$ and $2\alpha+\beta=0$ in $k$, we have
		
		\begin{equation}
			4\alpha^2(\alpha-\beta)= 4(\alpha^3+b)=4\left(-\frac{\alpha\cdot b}{ \beta}+b\right)=4\left( \frac{b}{2}+b\right)=6b,
		\end{equation} 
		so that $\left(\frac{\alpha-\beta}{k}\right)=\left(\frac{6b}{k}\right)$ and $w(E/\mathcal{K})=-\left(\frac{6b}{k}\right)$.
		
		For the remaining cases of potentially good reduction,  we show how the corollary follows from Theorem \ref{rootnos}$(vii)$. If $E/\mathcal{K}$ has type $\rm{II}$ then $v(\Delta)=2$ so 
		\begin{equation}
			w(E/\mathcal{K})=(-1)^{\big\lfloor\frac{|k|}{6} \big\rfloor}=\begin{cases}
				+1 &\text{ if } |k|\equiv 1 \text{ or } 5\pmod{12} \\
				-1 &\text{ if } |k|\equiv 7\text{ or } 11\pmod{12}
			\end{cases}=\left( \frac{-1}{k} \right).  
		\end{equation}
		A similar argument shows that if $E/\mathcal{K}$ has type $\rm{II}^*$, $w(E/\mathcal{K})=\left( \frac{-1}{k} \right)$. If $E/\mathcal{K}$ has type $\rm{IV}$, 
		\begin{equation}
			w(E/\mathcal{K})=(-1)^{\big\lfloor\frac{|k|}{3} \big\rfloor}=\begin{cases}
				+1 &\text{ if } |k|\equiv 1 \pmod{6} \\
				-1 &\text{ if } |k|\equiv 5\pmod{6}
			\end{cases}=\left( \frac{-3}{k} \right),
		\end{equation}
		and similarly if $E$ has type $\rm{IV}^*$, $w(E/\mathcal{K})=\left( \frac{-3}{k} \right)$, which completes the proof. 
	\end{proof}

	\begin{corollary}\label{semistable}
		Let $E$ be a semistable elliptic curve defined over a number field $K$. Then 
		\begin{equation}
			w(E/K)=(-1)^{m+u},
		\end{equation}
		where $m$ is the number of primes of $K$ where $E$ has split multiplicative reduction and $u$ is the number of infinite places of $K$.  
	\end{corollary}

	\begin{example}\label{semistableex}
		Let us take the curve
		\begin{equation}
			E:y^2+y=x^3-x^2. \hspace{10pt} (11A3)
		\end{equation}
		This curve has split multiplicative reduction at $11$ and good reduction everywhere else, so
		\begin{equation}
			w(E/\mathbb{Q})=w(E/\mathbb{Q}_{11})\cdot w(E/\mathbb{R})=(-1)(-1)=+1. 
		\end{equation}
		Assuming the parity conjecture, this tells us that the rank of $E/\mathbb{Q}$ is even. In fact, $E$ has rank $0$ over $\mathbb{Q}$. 
	\end{example}

	\begin{example}\label{firstquadtwistex}
		Let us take the curve 11A3 as in the above example. We know that $w(E/\mathbb{Q})=1$. Let us look at what happens to the global root number of $E$ over $K=\mathbb{Q}(\sqrt{-2})$. Since $-2$ is a quadratic residue modulo $11$, $11$ splits in $K$ into, say, $v$ and $\Tilde{v}$. Then $w(E/K_{v})=w(E/K_{\Tilde{v}})=-1$ since the reduction over $K_v$ and $K_{\Tilde{v}}$ is split multiplicative. Hence,
		\begin{equation} 
			w(E/K)=w(E/K_v)\cdot w(E/K_{\Tilde{v}})\cdot w(E/\mathbb{C})=(-1)^2\cdot (-1)=-1.
		\end{equation}
		Assuming the parity conjecture, this tells us that $E$ must acquire a point of infinite order over $K$.
	\end{example}
	
	Asking what happens to local root numbers when we look at the curve over an extension of the base field will inform many of our examples of curious parity phenomena in \S \ref{examples}. The above example illustrates the following key feature of root numbers of elliptic curves.
	
	\begin{lemma}\label{2primes}
		Let $E/K$ be an elliptic curve and let $F$ be a Galois extension of $K$. Let $v$ be a prime of $K$ and let $v_1$ and $v_2$ be primes above $v$ in $F$. Then 
		\begin{equation}
			w(E/F_{v_1})=w(E/F_{v_2}).
		\end{equation}
	\end{lemma}
	
	\begin{proof}
		Since $F/K$ is Galois, $F_{v_1}$ is isomorphic to $F_{v_2}$. The isomorphism from $F_{v_1}$ to $F_{v_2}$ preserves the equation for $E$ and the root number $w(E/F_{v_i})$ only depends on the isomorphism class of $E/F_{v_i}$. 
	\end{proof}

	\subsection{Root numbers of quadratic twists of elliptic curves}\label{rootnosquadtwists}
	If $E:y^2=f(x)$ is an elliptic curve defined over a number field $K$ and $d\in K^\times$, we call $E_d:dy^2=f(x)$ the quadratic twist of $E$ by $d$. Recall the following well-known relationship between the rank of $E$ and the rank of $E_d$. 
	
	\begin{lemma}\label{quadranks}
		Let $E$ be an elliptic curve defined over a number field $K$ and let $d\in K^\times\setminus K^{{\times 2}}$. Then 
		\begin{equation}
			{\rm{rk}}(E/K(\sqrt{d}))={\rm{rk}}(E/K)+{\rm{rk}}(E_d/K). 
		\end{equation}
	\end{lemma}
	
	\begin{lemma}\label{quadrootnos}
		Let $E/K$ and $E_d/K$ be as above. Then
		\begin{equation}
			w(E/K(\sqrt{d}))=w(E/K)\cdot w(E_d/K).
		\end{equation}
	\end{lemma}
	
	\begin{remark}
		We caution the reader that the statement of Lemma \ref{quadrootnos} is not true on the level of local root numbers, see \cite{cesna} Proposition 3.11 for the analogous statement in the local setting. The lemma follows from the local statement, along with the product formula for Hilbert symbols. 
	\end{remark}
	
	Lemma \ref{quadrootnos} tells us that the statement of the parity conjecture is consistent with what we know about the ranks of $E$ and $E_d$. One can exploit this relationship between global root numbers and use the parity conjecture to predict the existence of points of infinite order on quadratic twists of elliptic curves.

	\begin{example}
		Let us take $E$ to be the curve 11A3 over $\mathbb{Q}$ and let $K=\mathbb{Q}(\sqrt{-2})$. We saw in Examples~\ref{semistableex} and \ref{firstquadtwistex} that $w(E/\mathbb{Q})=+1$ and $w(E/K)=-1$. So, by Lemma \ref{quadrootnos}, $w(E_{-2}/\mathbb{Q})=-1$ and the parity conjecture implies there is a rational point of infinite order on $E_{-2}:y^2 = x^3-x^2-x-1$ over $\mathbb{Q}$. In fact, $E_{-2}$ has rank $1$ over $\mathbb{Q}$. 
	\end{example}

	\begin{theorem}\label{5050quad}
		Let $E$ be an elliptic curve over $\mathbb{Q}$. 
		\begin{enumerate}[(i)]
			\item The function $d\mapsto w(E_d/\mathbb{Q})$ is periodic on the set of positive (resp. negative) square-free integers. The period divides $\prod_{p|N_E}p^2$ if $N_E$ is odd and $4\prod_{p|N_E}p^2$ if $N_E$ is even; in particular, it divides $4N_E^2$.
			\item Root numbers of quadratic twists of $E$ satisfy
			\begin{equation}
				w(E_d/\mathbb{Q})=\begin{cases}
					+1 \text{ for } 50\% \text{ of square-free } d; \\ 
					-1 \text{ for } 50\% \text{ of square-free } d.
				\end{cases}
			\end{equation}
		\end{enumerate}
	\end{theorem}
	
	\begin{proof}
		These results are well-known, but we prove them here for lack of a reference. 
		
		$(i)$ Let $D=\prod_{p|N_E}p^2$ if $N_E$ is odd and $D=4\prod_{p|N_E}p^2$ if $N_E$ is even. Suppose $d_1\equiv d_2\pmod{D}$ and $d_1$ and $d_2$ have the same sign, and let $p\mid N_E$. First suppose $p>2$. Since $d_1\equiv d_2\pmod{p^2}$ and $d_1$ and $d_2$ are square-free, it follows that $d_1/d_2\equiv 1\pmod{p}$, so $d_1/d_2$ is a square in $\mathbb{Q}_p$ by Hensel's lemma. So $\mathbb{Q}_p(\sqrt{d_1})=\mathbb{Q}_p(\sqrt{d_2})$ and for a prime $v\mid p$ in $\mathbb{Q}(\sqrt{d_1})$ and $\Tilde{v}$ in $\mathbb{Q}(\sqrt{d_2})$, $w(E/\mathbb{Q}(\sqrt{d_1})_v)=w(E/\mathbb{Q}(\sqrt{d_2})_{\Tilde{v}})$. Since $p$ has the same splitting behaviour in $\mathbb{Q}(\sqrt{d_1})$ and $\mathbb{Q}(\sqrt{d_2})$,
		\begin{equation}
			\prod_{v\mid p} w(E/\mathbb{Q}(\sqrt{d_1})_v)= \prod_{\Tilde{v}\mid p} w(E/\mathbb{Q}(\sqrt{d_2})_{\Tilde{v}}). 
		\end{equation}
		Now, if $N_E$ is even and $p=2$, $d_1\equiv d_2\pmod{16}$ and since $d_1$ and $d_2$ are square-free, it follows that $d_1/d_2\equiv 1\pmod{8}$ and $d_1/d_2$ is a square in $\mathbb{Q}_2$. Thus, $\prod_{v\mid 2} w(E/\mathbb{Q}(\sqrt{d_1})_v)= \prod_{\Tilde{v}\mid 2} w(E/\mathbb{Q}(\sqrt{d_2})_{\Tilde{v}})$ by the same argument as above. Finally, note that $\mathbb{Q}(\sqrt{d_1})$ and $\mathbb{Q}(\sqrt{d_2})$ have the same number of infinite places. Hence
		\begin{equation}
			w(E/\mathbb{Q}(\sqrt{d_1}))=(-1)^{v\mid \infty}\prod_{v\mid N_E} w(E/\mathbb{Q}(\sqrt{d_1})_v) = (-1)^{\Tilde{v}\mid \infty}\prod_{\Tilde{v}\mid N_E} w(E/\mathbb{Q}(\sqrt{d_2})_{\Tilde{v}})=w(E/\mathbb{Q}(\sqrt{d_2})).
		\end{equation}
		Thus $w(E_{d_1}/\mathbb{Q})=w(E_{d_2}/\mathbb{Q})$ by Lemma \ref{quadrootnos}.
		
		$(ii)$ First, we claim that we can find a $d_0<0$ such that $(d_0,2 N_E)=1$ and, for every $d$, we have 
		\begin{equation}
			w(E_{dd_0}/\mathbb{Q})=-w(E_d/\mathbb{Q}). 
		\end{equation}
		Indeed, choose $d_0<0$ such that all the primes of bad reduction of $E$ split in $\mathbb{Q}(\sqrt{d_0})$ and $(d_0,2 N_E)=1$. By Lemma \ref{2primes}, this ensures that $w(E/\mathbb{Q}(\sqrt{d_0}))=-1$. If $d=d_0$, by Lemma \ref{quadrootnos}
		\begin{equation}
			w(E_{d_0^2}/\mathbb{Q})=w(E/\mathbb{Q})=w(E/\mathbb{Q}(\sqrt{d_0}))w(E_{d_0}/\mathbb{Q})=-w(E_{d_0}/\mathbb{Q}). 
		\end{equation}
		If $d\neq d_0$, by Lemma \ref{quadrootnos} and Lemma \ref{2primes},
		\begin{align}
			w(E/\mathbb{Q})\cdot w(E_{d_0}/\mathbb{Q})\cdot w(E_{d}/\mathbb{Q})\cdot w(E_{dd_0}/\mathbb{Q})=w(E/\mathbb{Q}(\sqrt{d},\sqrt{d_0}))=+1
		\end{align}
		since $\mathbb{Q}(\sqrt{d},\sqrt{d_0})$ has two infinite places and all finite places dividing the discriminant of $E$ split into an even number of places in $\mathbb{Q}(\sqrt{d},\sqrt{d_0})$. Now, 
		\begin{align}
			w(E/\mathbb{Q})\cdot w(E_{d_0}/\mathbb{Q}) =w(E/\mathbb{Q}(\sqrt{d_0}))=-1.
		\end{align}
		Hence $w(E_{d}/\mathbb{Q})\cdot w(E_{dd_0}/\mathbb{Q})=-1$ and we have proved the claim.
		
		Now, to finish proving the theorem, take $d_0$ as above and $D$ as in the proof of $(i)$. By part $(i)$ of the theorem, the function $d\mapsto w(E_d/\mathbb{Q})$ is periodic on the set of positive (resp. negative) square-free integers with the period dividing $D$. The involution $d\leftrightarrow d\cdot d_0$ changes the sign of the root number i.e. for any $d$ such that $w(E_d/\mathbb{Q})=+1$, we have $w(E_{d\cdot d_0}/\mathbb{Q})=-1$. Since $d_0$ is coprime to $D$, the density of square-free integers congruent to $d$ modulo $D$ is the same as the density of those congruent to $dd_0$ modulo $D$, see e.g.~\cite{stackexchange}. The result follows. 
	\end{proof} 
	
	Theorem \ref{5050quad} gives a heuristic for Goldfeld's conjecture, which states that there is a $50/50$ distribution of the rank being $0$ and $1$ in a quadratic twist family over $\mathbb{Q}$, see \cite{goldfeld} Conjecture B. Over general number fields, root numbers of quadratic twists are also periodic (by essentially the same proof). However, the distribution of the root numbers can be different. As an extreme case, it is possible for all quadratic twists to have the same root number; see \S\ref{nogoldfeld5} or \cite{posquadtwists} for the original discussion. In \cite{klagsbrun} Example 7.11 and \cite{zev2}, it is shown that the set of proportions that appear for elliptic curves is dense in $[0,1]$. In \cite{klagsbrun} and \cite{zev2} this is phrased in terms of $2$-Selmer ranks --- this is equivalent to the corresponding statement for global root numbers because the $2$-parity conjecture is known in this instance by \cite{regconst} Theorem 1.3.
	
	\begin{remark}
		We expect that for elliptic curves over $\mathbb{Q}$ in a suitable ordering, there is a $50/50$ distribution of the global root number being $+1$ and $-1$ (for the ordering used, see Notation \ref{ordering}). Despite the fact that this is a statement about root numbers rather than about ranks, this remains an open problem.
	\end{remark}
	
	\subsection{Root numbers of Artin twists of elliptic curves}\label{paritytwists}
	
	Let $E/K$ be an elliptic curve and let $F$ be a Galois extension of $K$. Then $E(F)\otimes_\mathbb{Z}\mathbb{C}$ is naturally a $\text{Gal}(F/K)$-representation by letting elements of the Galois group act on the co-ordinates of the $F$-rational points of $E$. There is a parity conjecture for Artin twists of elliptic curves which describes the parity of the multiplicity of an irreducible, self-dual $\text{Gal}(F/K)$-representation $\rho$ in $E(F)\otimes_\mathbb{Z}\mathbb{C}$. Since $\text{rk}(E/F)=\text{dim}E(F)\otimes_\mathbb{Z}\mathbb{C}$, this can tell us about the existence of rational points over $F$. The parity conjecture for twists is especially striking when the usual parity conjecture does not provide us with any information about the existence of rational points (see Example~\ref{twistsexample}). We state the parity conjecture for Artin twists here and present a theorem that allows us to explicitly calculate global root number of twists of elliptic curves over $\mathbb{Q}$ by self-dual Artin representations when the conductor of $\rho$ is coprime to the conductor of $E$. 
	
	\begin{paritytwists}
		Let $E/K$ be an elliptic curve over a number field and let $\rho$ be a self-dual Artin representation of $\textup{Gal}(\bar{K}/K)$ that factors through $\emph{\text{Gal}}(F/K)$, for a finite extension $F/K$. Then 
		\begin{equation}
			(-1)^{\langle \rho,E(F)\otimes_{\mathbb{Z}} \mathbb{C}\rangle}=w(E/K,\rho),
		\end{equation}
		where $w(E/K,\rho)$ is the global root number of the twist of $E$ by $\rho$. 
	\end{paritytwists}
	
	\begin{theorem}[\cite{vladtwists} Corollary 2]\label{rootnostwists}
		Let $\rho$ be a self-dual Artin representation of $\textup{Gal}(\bar{\mathbb{Q}}/\mathbb{Q})$. Let $E$ be an elliptic curve over $\mathbb{Q}$ whose conductor $N_E$ is coprime to the conductor of $\rho$. Then
		\begin{equation}
			w(E/\mathbb{Q},\rho)= w(E/\mathbb{Q})^{\emph{\text{dim }}\rho}\cdot\emph{\text{sign}}(\alpha_\rho)\cdot\left(\frac{\alpha_\rho}{N_E}\right)
		\end{equation}
		where $\left(\frac{*}{*}\right)$ is the Jacobi symbol and $\alpha_\rho=1$ if $\emph{\text{det}}(\rho)=\mathds{1}$ and, otherwise, $\alpha_\rho$ is such that the character $\emph{\text{det}}(\rho)$ factors through $\mathbb{Q}(\sqrt{\alpha_\rho})$. We adopt the convention that $\left(\frac{\alpha_\rho}{2}\right)=-1$ if $\alpha_\rho\equiv 5\pmod{8}$ and $\left(\frac{\alpha_\rho}{2}\right)=+1$ if $\alpha_\rho\equiv 1 \pmod{8}$. 
	\end{theorem}
	
	Note that if $2\mid N_E$ then $2$ does not ramify in $\mathbb{Q}(\sqrt{\alpha_\rho})$ since we assumed the conductor of $\rho$ to be coprime to the conductor of $E$, and hence $\alpha_\rho\equiv 1 \text{ or } 5 \pmod{8}$. For a statement that allows multiplicative reduction at primes dividing the conductor of $\rho$, see \cite{vladtwists} Theorem 1. Local root numbers of Artin twists of elliptic curves have been classified by Rohrlich (with restrictions in residue characteristic $2$ and $3$), see \cite{rohrlich} Theorem 2. 
	
	We present the following important properties of global root numbers of twists, which are also of use when predicting the existence of rational points of infinite order using the parity conjecture for twists (see Example \ref{twistsexample}). For reference, $(i)$ and $(ii)$ follow from the definition of local root numbers (see \cite{deligne} for more details) and $(iii)$ and $(iv)$ can be found in \cite{regconst} Proposition A.2.
	
	\begin{theorem}\label{artinformalism}
		Let $E$ be an elliptic curve over a number field $K$. Let $\rho$ and $\rho'$ be Artin representations of $G_K=\textup{Gal}(\bar{K}/K)$. Let $L/K$ be a finite extension and let $\tau$ be an Artin representation of $G_L=\textup{Gal}(\bar{L}/L)$. Then 
		\begin{enumerate}[(i)]
			\item $w(E/K,\mathds{1})=w(E/K)$;
			\item $w(E/K,\rho\oplus\rho')=w(E/K,\rho)\cdot w(E/K,\rho')$;
			\item $w(E/K,\emph{\text{Ind}}_{G_L}^{G_K}\tau)=w(E/L,\tau)$ and, in particular, $w(E/L)=w(E/K,\emph{\text{Ind}}_{{G_L}}^{G_K} \mathds{1})$;
			\item $w(E/K,\rho\oplus\rho^*)=+1$, where $\rho^*$ is the dual representation.
		\end{enumerate}
	\end{theorem}
	
	We can relate the rank over subfields of $F$ to the representation $E(F)\otimes_\mathbb{Z}\mathbb{C}$ using the following lemma. 
	
	\begin{lemma}\label{frob}
		Let $E$ be an elliptic curve over $K$ and let $F$ be a Galois extension with $G=\textup{Gal}(F/K)$. Then for every $H\leq G$,
		\begin{equation}
			\textup{rk}(E/F^H)=\langle\mathbb{C}[G/H],E(F)\otimes_\mathbb{Z}\mathbb{C}\rangle_G. 
		\end{equation}
	\end{lemma}
	
	\begin{proof}
		For a subgroup $H\leq G$, we have $\textup{rk}(E/F^H)={\textup{{dim}}} (E(F)\otimes_\mathbb{Z}\mathbb{C})^H$ and 
		\begin{equation}
			{\textup{{dim}}} (E(F)\otimes_\mathbb{Z}\mathbb{C})^H =\langle \mathds{1},\textup{Res}_H^G E(F)\otimes_\mathbb{Z}\mathbb{C}\rangle_H 
			\stackrel{(1)}{=}\langle\textup{Ind}_{H}^G\mathds{1},E(F)\otimes_\mathbb{Z}\mathbb{C}\rangle_G=\langle\mathbb{C}[G/H],E(F)\otimes_\mathbb{Z}\mathbb{C}\rangle_G,
		\end{equation}
		where $(1)$ follows from Frobenius reciprocity. 
	\end{proof}
	
	The following lemma tells us that in many cases (for instance symmetric groups) the parity conjecture for twists follows from the usual parity conjecture. 
	
	\begin{lemma}\label{permutation}
		Let $E/K$ be an elliptic curve and let $F/K$ be a Galois extension of number fields. Let $\rho$ be a representation of $G=\textup{Gal}(F/K)$ that can be written as a linear combination of permutation modules i.e.
		\begin{equation}
			\rho\oplus\bigoplus_{i}\textup{Ind}_{H_i}^G\mathds{1}\simeq \bigoplus_{j}\textup{Ind}_{H_j'}^G\mathds{1}
		\end{equation}
		for some $H_i$, $H_j'\leq G$. If the parity conjecture holds for $E$ over $L_i=F^{H_i}$ for all $i$ and $L_j'=F^{H_j'}$ for all $j$, then the twisted parity conjecture holds for $E$ and $\rho$.
	\end{lemma}
	
	\begin{proof}
		By Theorem \ref{artinformalism} and assuming the usual parity conjecture for $L_i$ and $L_j$,
		\begin{equation}
			w(E/K,\rho)=\frac{\prod_j w(E/K,\textup{Ind}_{H_j'}^G\mathds{1})}{\prod_i w(E/K,\textup{Ind}_{H_i}^G\mathds{1})}=\frac{\prod_jw(E/L_j)}{\prod_i w(E/L_i)}=(-1)^{\sum_j \text{rk}(E/L_j)-\sum_i\text{rk}(E/L_i)}.
		\end{equation}
		By Lemma \ref{frob},
		\begin{align}
			\sum_j \text{rk}(E/L_j)-\sum_i\text{rk}(E/L_i)
			=\sum_j \langle E(F)\otimes_\mathbb{Z}\mathbb{C},\textup{Ind}_{H_j'}^G\mathds{1}\rangle-\sum_i \langle E(F)\otimes_\mathbb{Z}\mathbb{C},\textup{Ind}_{H_i}^G\mathds{1}\rangle = \langle E(F)\otimes_\mathbb{Z}\mathbb{C},\rho \rangle. 
		\end{align}
		Thus the parity conjecture for twists holds for $E$ and $\rho$.
	\end{proof}
	
	When the condition of Lemma \ref{permutation} is not satisfied, the parity conjecture for twists gives us more information about the arithmetic of elliptic curves than the usual parity conjecture, as we now illustrate. 
	
	\begin{example}\label{twistsexample}
		Let us take the elliptic curve of rank $1$ over $\mathbb{Q}$ 
		\begin{equation}
			E:y^2+y=x^3-x. \hspace{10pt} (37A1)
		\end{equation}
		Let $F$ be the splitting field of the polynomial
		\begin{equation}
			x^{10} - x^9 + 6x^8 - 3x^7 + 11x^6 - 3x^5 + 11x^4 - 3x^3 + 6x^2 - x + 1,
		\end{equation}
		which has discriminant $-47^5$. Then $G=\text{Gal}(F/\mathbb{Q})=D_{10}$, the dihedral group with $10$ elements. We can predict the existence of points of infinite order on $E/F$ using the parity conjecture for twists as follows. 
		
		The subgroups of $D_{10}$, up to conjugacy, are the trivial group, $C_2$, $C_5$ and $D_{10}$. The irreducible representations of $D_{10}$ are the trivial representation $\mathds{1}$, sign representation $\epsilon$, and two $2$-dimensional representations $\rho_1$ and $\rho_2$. Now let $V=E(F)\otimes_\mathbb{Z}\mathbb{C}$. This is a $G$-representation so we can decompose it as
		\begin{equation}
			V=\mathds{1}^{\oplus a}\oplus\epsilon^{\oplus b}\oplus\rho_1^{\oplus c}\oplus\rho_2^{\oplus d}.
		\end{equation}
		Since the character of this representation is rational, $\rho_1$ and $\rho_2$ must appear with the same multiplicity as their characters are $G_\mathbb{Q}$-conjugate. So we have $V=\mathds{1}^{\oplus a}\oplus\epsilon^{\oplus b}\oplus\rho_1^{\oplus c}\oplus \rho_2^{\oplus c}$. By Lemma \ref{frob}, for any subgroup $H\leq \text{Gal}(F/\mathbb{Q})$,
		\begin{equation}
			\text{rk}(E/F^H)=\text{dim} V^H=\langle\mathbb{C}[G/H],V\rangle_G.
		\end{equation}
		Observe that 
		\begin{equation}
			\mathbb{C}[D_{10}/D_{10}]=\mathds{1}, \hspace{20pt}
			\mathbb{C}[D_{10}/C_5]=\mathds{1}\oplus\epsilon,\hspace{20pt}
			\mathbb{C}[D_{10}/C_2]=\mathds{1}\oplus\rho_1\oplus\rho_2,\hspace{20pt}
			\mathbb{C}[D_{10}]=\mathds{1}\oplus\epsilon\oplus\rho_1^{\oplus2}\oplus\rho_2^{\oplus2}. 
		\end{equation}
		Hence,
		\begin{equation}
			\text{rk}(E/F^H)=\begin{cases}
				a & \text{ when } H=D_{10}; \\
				a+b & \text{ when } H=C_5; \\
				a+2c & \text{ when } H=C_2; \\
				a+b+4c & \text{ when } H=\{1\}.
			\end{cases}
		\end{equation}
		$E$ has non-split multiplicative reduction at $37$ and good reduction everywhere else. Since $37$ splits in $F^{C_5}=\mathbb{Q}(\sqrt{-47})$, by Corollary \ref{semistable} we find that
		\begin{equation}
			w(E/\mathbb{Q})=-1 \quad \text{and} \quad w(E/\mathbb{Q}(\sqrt{-47}))=-1.
		\end{equation}
		So the parity conjecture and our calculations above imply that the rank of $E$ is odd over all four intermediate fields. In particular, we could in principle have $\text{rk}(E/F)=\text{rk}(E/\mathbb{Q})$.
		
		We now use the parity conjecture for twists to show that $c$ is odd, and hence $\text{rk}(E/F)\geq \text{rk}(E/\mathbb{Q})+4\geq 5$. Using the notation of Theorem \ref{rootnostwists}, we have $N_E=37$ and $\alpha_{\rho_1}=-47$ since $\text{det}(\rho_1)=\epsilon$. Hence, by Theorem \ref{rootnostwists}, 
		\begin{equation}
			w(E/\mathbb{Q},\rho_1)=(-1)^2\cdot (-1)\cdot \left(\frac{-47}{37} \right)=-1.
		\end{equation}
		Thus, since $\langle  \rho_1,E(F)\otimes_\mathbb{Z}\mathbb{C}\rangle=c$, the parity conjecture for twists tells us that $c$ must be odd and hence $\text{rk}(E/F)\geq 5$. In fact, assuming the parity conjecture for twists, the rank of $E/F$ is larger than the rank of $E$ over any of the subfields.
	\end{example}

	\section{Parity phenomena}\label{examples}
	
	We now turn to examples of phenomena that are predicted by the parity conjecture as a result of root number calculations. Subsections \ref{towersection} and \ref{sqrtK}--\ref{evenartin} deal with consequences of the parity conjecture for Artin twists. We stress that all `parity phenomena' that we describe, including the statements in all of the subsection titles, are conditional on the parity conjecture and are effectively unsolved problems. Most of these have been observed before; only \S\ref{represented}, \S\ref{sqrtK} and \S\ref{snsection} are new. The titles are designed to give specific examples of what will be discussed; the general statements will be contained within the subsections. The subsections are largely independent of one another and can be read in almost any order. 
	
	\subsection{Every $E/\mathbb{Q}$ has even rank over $\mathbb{Q}(\sqrt{-3},\sqrt{13})$}\label{parityex}
	
	\begin{example}[See also \cite{timparity}, Conjecture 8.7]\label{17}
		The parity conjecture predicts that every elliptic curve $E/\mathbb{Q}$ has even rank over $K=\mathbb{Q}(\sqrt{-3},\sqrt{13})$. To see this, we note that every rational prime splits into an even number of primes in $K$. Indeed, the decomposition group of rational primes away from the discriminant of $K$ is cyclic and so cannot be isomorphic to $C_2\times C_2$. Hence these primes split into $2$ or $4$ primes in $K$. The primes that divide the discriminant of $K$ are $3$ and $13$, and we note that $3$ splits in $\mathbb{Q}(\sqrt{13})$ and $13$ splits in $\mathbb{Q}(\sqrt{-3})$. So, $3$ and $13$ each split into two primes in $K$. Thus,
		\begin{equation}
			w(E/K)=\prod_{v\mid\infty}w(E/K_v)\cdot\prod_{v\mid N_E}w(E/K_v) =w(E/\mathbb{C})^2\cdot\prod_{p\mid N_E}\prod_{v\mid p}w(E/K_v)=(-1)^2\cdot\prod_{p\mid N_E}(\pm 1)^{2\text{ or }4}=+1,
		\end{equation}
		where $v$ runs over primes of $K$ and $p$ runs over primes of $\mathbb{Q}$. Here we have used the fact that if $v$ and $v'$ lie above the same prime $p$, then $w(E/K_v)=w(E/K_{v'})$ by Lemma \ref{2primes}. Assuming the parity conjecture, we deduce that $E$ has even rank over $K$.
	\end{example}
	
	The same reasoning gives us the following lemma. 
	
	\begin{lemma}\label{even}
		Let $E/K$ be an elliptic curve and let $F$ be a finite Galois extension of $K$ in which every place of $K$ splits into an even number of places. Then, assuming the parity conjecture, $E$ has even rank over $F$. 
	\end{lemma}
	
	\begin{example}\label{3,13}
		Let 
		\begin{equation}
			E:y^2+xy=x^3-x^2+4x-3. \hspace{10pt}\text{ \text{(73A1)} }
		\end{equation}
		Since $E$ has split multiplicative reduction at $73$, Corollary \ref{semistable} tells us that $w(E/\mathbb{Q})=+1$. This curve has a fairly obvious point over $K=\mathbb{Q}(\sqrt{-3},\sqrt{13})$, namely $(0,\sqrt{-3})$, which turns out to have infinite order. Example \ref{17} tells us that $\text{rk}(E/K)$ is even, so there should be another (independent) point of infinite order over $K$. In effect, the existence of a point of infinite order over $\mathbb{Q}(\sqrt{-3})$ forces the existence of another point of infinite order over $\mathbb{Q}(\sqrt{-3})$, $\mathbb{Q}(\sqrt{13})$ or $\mathbb{Q}(\sqrt{-39})$. In this case, local root number calculations show that $w(E/\mathbb{Q}(\sqrt{13}))=-1$, so, assuming the parity conjecture, there should also be a $K$-rational point of infinite order coming from points on $E$ over $\mathbb{Q}(\sqrt{13})$. Indeed, the point $\left(3,-\frac{3\sqrt{13}+3}{2}\right)$ lies on $E$ and generates the infinite part of the Mordell-Weil group of $E/\mathbb{Q}(\sqrt{13})$. It is not at all clear how the points over $\mathbb{Q}(\sqrt{13})$ relate to those over $\mathbb{Q}(\sqrt{-3})$, and why the existence in one of the fields guarantees the existence in the other. 
	\end{example}

	\subsection{Every $E/\mathbb{Q}$ has even rank over any $C_2\times C_2 \times C_2 \times C_2$ extension of $\mathbb{Q}$}\label{c2c2}
	
	Lemma \ref{even} also gives us the following example. 
	
	\begin{example}\label{c2}
		Let $K$ be any Galois extension of $\mathbb{Q}$ with Galois group $C_2^d$, where $d\geq 4$ and let $E$ be an elliptic curve defined over $\mathbb{Q}$. We claim that $E$ should have even rank over $K$. As in Example \ref{17}, all unramified primes split into an even number of primes in $K$. We claim that every ramified prime splits into an even number of primes in $K$, too. If $p>2$, note that the largest extension of $\mathbb{Q}_p$ we can get when we localise at a prime $v$ above $p$ in $K$ has Galois group $C_2\times C_2$ since $\mathbb{Q}_p^\times/{\mathbb{Q}_p^{\times 2}}\cong C_2\times C_2$ for $p>2$. Hence the decomposition group of $v\mid p$ has size at most $4$ and so there must be at least four primes above $p$ in $K$. Similarly for $p=2$, the largest extension of $\mathbb{Q}_2$ we can get when we localise at a prime $v$ above $2$ has Galois group $C_2\times C_2\times C_2$ and so the decomposition group of $v\mid 2$ has size at most $8$ and there are at least two primes above $2$ in $K$. Since $K$ has an even number of infinite places, Lemma \ref{even} tells us that, assuming the parity conjecture, $E$ has even rank over $K$.
	\end{example}
	
	\begin{remark}
		The parity conjecture implies that if $E/\mathbb{Q}$ has a point of infinite order over a $C_2\times C_2\times C_2\times C_2$-extension of $\mathbb{Q}$, it must automatically acquire a second one. This suggests that over any $C_2\times C_2\times C_2\times C_2$-extension of $\mathbb{Q}$, there might be some as yet unknown extra symmetry, for instance in the Mordell--Weil group, that would explain why the rank is even. As Example \ref{3,13} illustrates, it is not clear what such a symmetry might be. 
	\end{remark}

	\subsection{$y^2+xy=x^3+Ax+B$ where $A\equiv B\pmod{2}$ has infinitely many solutions over $\mathbb{Q}(\zeta_8)$}
	
	\begin{example}\label{8}
		The parity conjecture predicts that every elliptic curve $E/\mathbb{Q}$ with split multiplicative reduction at $2$ has infinitely many rational points over $K=\mathbb{Q}(\zeta_8)$. Here, the only ramified prime is $2$, which is totally ramified. So, if $E$ has split multiplicative reduction at $2$, for the unique prime $v$ above $2$ in $K$ we have $w(E/K_v)=-1$ by Theorem \ref{rootnos}. Note that $\text{Gal}(K/\mathbb{Q})\cong C_2\times C_2$, so in the same vein as Example \ref{17}, all primes away from $2$ split into an even number of primes in $K$, which means their product contributes a $+1$ to the root number. There are two infinite places and so these also contribute a $+1$ to the root number. Hence, $w(E/K)=-1$ and the parity conjecture predicts that $E$ has odd rank over $K$. In particular, it says that $E$ has infinitely many $\mathbb{Q}(\zeta_8)$-rational points. We emphasise that it is not at all clear how to construct these points!
	\end{example}
	
	The above working tells us that every elliptic curve over $\mathbb{Q}$ of the form $E:y^2+xy=x^3+Ax+B$, where $A\equiv B\pmod{2}$, has infinitely many points over $\mathbb{Q}(\zeta_8)$ since every such curve has split multiplicative reduction at $2$. The example easily generalises to the following statement. 
	
	\begin{lemma}
		Suppose $K/\mathbb{Q}$ is a biquadratic extension and that exactly one prime $p$ does not split in $K/\mathbb{Q}$. If the parity conjecture holds, then every elliptic curve $E/\mathbb{Q}$ with split multiplicative reduction at $p$ has a point of infinite order over $K$. 
	\end{lemma}
	
	\begin{example}
		The curves 
		\begin{enumerate}[(i)]
			\item $E:y^2+xy=x^3-3x+1$ (34A1), 
			\item $E':y^2+xy=x^3-34x+68$ (102B1) and  
			\item $E'':y^2+xy=x^3-2x-2$ (922A1)
		\end{enumerate}
		each have split multiplicative reduction at $2$, so by Example \ref{8}, the parity conjecture implies that these curves all have infinitely many points over $\mathbb{Q}(\zeta_8)$. In fact, all of these curves have global root number $+1$ over $\mathbb{Q}$, so the rank must grow in $\mathbb{Q}(\zeta_8)$. One can perform descent calculations to show that the rank of each of these curves over $\mathbb{Q}(\zeta_8)$ is $1$. However, the point comes from a different quadratic subfield for each curve: $E$ has a point of infinite order over $\mathbb{Q}(\sqrt{2})$, $E'$ has a point of infinite order over $\mathbb{Q}(\sqrt{-1})$ and $E''$ has a point of infinite order over $\mathbb{Q}(\sqrt{-2})$.
	\end{example}
	
	\begin{remark}[See also \cite{modn} Remark 4]
		Consider the elliptic curve $E:y^2=x^3+x^2-x$, which has Kodaira type $\rm{IV}$ over $\mathbb{Q}_2$. By computing root numbers of the quadratic twists of $E$ by $-1$, $2$ and $-2$ and using Lemma \ref{quadrootnos}, we find that $w(E/\mathbb{Q}(\zeta_8))=-1$. The parity conjecture predicts that $E$ has odd rank over $\mathbb{Q}(\zeta_8)$ in a similar way to Example \ref{8}. The fact that $E/\mathbb{Q}(\zeta_8)$ has odd analytic rank has the following consequence for $L$-functions. 
		
		Since all primes away from $2$ split in $\mathbb{Q}(\zeta_8)$ and $E$ has additive reduction at the prime above $2$ in $\mathbb{Q}(\zeta_8)$, the Euler product of the $L$-function of $E/\mathbb{Q}(\zeta_8)$ is formally a square, in the sense that each Euler factor appears an even number of times:
		\begin{equation}
			L(E/\mathbb{Q}(\zeta_8),s)=1\cdot\frac{1}{(1+2\cdot9^{-s}+9^{1-2s})^2}\frac{1}{(1+25^{-s})^2} \frac{1}{(1+10\cdot 49^{-s}+49^{1-2s})^2}\cdots.
		\end{equation}
		If we define $F(s)$ by the square root of this Euler product so that $F(s)^2=L(E/\mathbb{Q}(\zeta_8),s)$ for $\text{Re}(s)>\frac{3}{2}$, then $F(s)$ shares many properties that one would expect from an $L$-function. However, it does not have analytic continuation to $s=1$ since the order of vanishing of $L(E/\mathbb{Q}(\zeta_8),s)$ at $s=1$ is odd.
	\end{remark}

	\subsection{$y^2+y=x^3+x^2+x$ has infinitely many $\mathbb{Q}(\sqrt[3]{m})$-rational solutions for every $m$}
	
	\begin{example}[As considered in \cite{vladtwists}]\label{19}
		The elliptic curve 
		\begin{equation}
			E: y^2+y=x^3+x^2+x \hspace{10pt}\text{(19A3)}
		\end{equation}
		has rank $0$ over $\mathbb{Q}$. We claim that, assuming the parity conjecture, it has infinitely many $\mathbb{Q}(\sqrt[3]{m})$-rational solutions for every cube-free $m>1$. First, note that $\Delta_E=19$, and $E$ has split multiplicative reduction at~$19$. Hence, $w(E/K_{v})=-1$ at primes $v$ above $19$. So, to calculate the global root number, all we need to know is how many primes there are above $19$ in $\mathbb{Q}(\sqrt[3]{m})$. If $19\nmid m$, the Kummer-Dedekind theorem tells us that there is either one or three primes above $19$, corresponding to whether $X^3-m$ is irreducible or splits completely over $\mathbb{F}_{19}$. If $19\mid m$, it is totally ramified. The product of local root numbers at the infinite places is $+1$ since $\mathbb{Q}(\sqrt[3]{m})$ has one real embedding and one complex embedding. This tells us that 
		\begin{equation}
			w(E/\mathbb{Q}(\sqrt[3]{m}))=(-1)^{1 \text{ or } 3}=-1.
		\end{equation}
		So, assuming the parity conjecture, $E$ must have infinitely many $\mathbb{Q}(\sqrt[3]{m})$-rational points. 
	\end{example}
	
	We can numerically compute the points of infinite order on $E$ over $\mathbb{Q}(\sqrt[3]{m})$ for small $m$ (for $m=1,\dots,7$, $E$ has rank $1$ over $\mathbb{Q}(\sqrt[3]{m})$).
	
	\begin{center}
		\begin{tabular}{ |c|c| } 
			\hline
			$m$ &  Generator of $E(\mathbb{Q}(\sqrt[3]{m}))/E(\mathbb{Q}(\sqrt[3]{m}))_{\text{tors}}$ \\ 
			\hline 
			& \\
			$2$ & \hspace{15pt}$\left(2\sqrt[3]{2}^2 - \sqrt[3]{2}\quad,\quad -2\sqrt[3]{2}^2 + \sqrt[3]{2} + 5\right)$ \\ 
			& \\
			$3$ & \hspace{9pt}$\left(\frac{13\sqrt[3]{3}^2 + 6\sqrt[3]{3} - 3}{25}\quad,\quad \frac{-52\sqrt[3]{3}^2 - 24\sqrt[3]{3} - 213}{125}\right)$ \\ 
			& \\
			$4$ & $\left(\frac{-\sqrt[3]{4}^2 + 4\sqrt[3]{4}}{2}\quad,\quad\frac{\sqrt[3]{4}^2 - 4\sqrt[3]{4} + 10}{2}\right)$ \\ 
			& \\
			$5$ & \hspace{11pt}$\left(\frac{82\sqrt[3]{5}^2 - 144\sqrt[3]{5} - 160}{529}\quad,\quad \frac{-574\sqrt[3]{5}^2 +
				1008\sqrt[3]{5} - 6815}{12167}\right)$ \\ 
			& \\ 
			$6$ & \hspace{16pt}$\left(\frac{7340\sqrt[3]{6}^2 - 8885\sqrt[3]{6} - 9450}{90774}\quad,\quad
			\frac{748680\sqrt[3]{6}^2 - 906270\sqrt[3]{6} - 992477}{11165202}\right)$ \\ 
			& \\
			$7$ & \hspace{25pt}$\left(\frac{50\sqrt[3]{7}^2 - 40\sqrt[3]{7}}{63} \quad,\quad \frac{-150\sqrt[3]{7}^2 + 120\sqrt[3]{7} + 811}{189} \right)$ \\ 
			& \\
			\hline
		\end{tabular}
	\end{center}
	
	\vspace{10pt}
	
	Without the parity conjecture, proving the existence of these points for all $m$ seems completely out of reach. In \cite{timcubic}, it is shown that for an elliptic curve $E$ over a number field $K$, the rank of $E$ goes up in infinitely many extensions of $K$ obtained by adjoining a cube root of an element of $K$. However, these points account for very few $m$ in Example \ref{19}. 
	
	\begin{remark}
		We caution the reader that in the above example there are no parametric solutions of the form 
		\begin{equation}
			f:m\mapsto P_m\in E(\mathbb{Q}(\sqrt[3]{m})).
		\end{equation}
		Indeed, if $f$ were analytic, it would give an analytic map $\mathbb{P}^1\rightarrow  E(\mathbb{C})$ which contradicts the Riemann--Hurwitz formula. If $f$ were only assumed to be continuous, it would give a map $\mathbb{R}\rightarrow E(\mathbb{R})$. For $m\in \mathbb{Q}^{\times 3}$, this could only take values $(0,0)$, $(0,1)$ or $\mathcal{O}$, which would force $f$ to be constant.  
	\end{remark}
	
	\subsection{$E:y^2+y=x^3+x^2+x$ has rank at least $n$ over $\mathbb{Q}(\sqrt[3^n]{m})$}\label{firsttower}
	
	\begin{example}\label{19ext}
		Example \ref{19} tells us that the rank of $E:y^2+y=x^3+x^2+x$ (19A3) must grow over~$\mathbb{Q}(\sqrt[3]{m})$. Assuming the parity conjecture, it grows at every step of the tower $\big(\mathbb{Q}(\sqrt[3^n]{m})\big)_{n\geq 1}$. To see this, we claim that $19$ splits into an odd number of places in $\mathbb{Q}(\sqrt[3^n]{m})$. Indeed, $\mathbb{Q}_{19}(\zeta_{3^n},\sqrt[3^n]{m})$ is an odd degree extension of $\mathbb{Q}_{19}$ and so $e_{\Tilde{v}} f_{\Tilde{v}}$ is odd for every prime $\Tilde{v}$ above $19$ in $\mathbb{Q}(\zeta_{3^n},\sqrt[3^n]{m})$, where $e_{\Tilde{v}}$ is the ramification degree and $f_{\Tilde{v}}$ is the residue degree of $\Tilde{v}$ over $19$. Let $v$ be a prime above $19$ in $\mathbb{Q}(\sqrt[3^n]{m})$. Since $e_v f_v$ divides $ e_{\Tilde{v}} f_{\Tilde{v}}$, $e_v f_v$ is also odd. Then, since $[\mathbb{Q}(\sqrt[3^n]{m}):\mathbb{Q}]$ is odd and $[\mathbb{Q}(\sqrt[3^n]{m}):\mathbb{Q}]=\sum_{i=1}^{k} e_{v_i} f_{v_i}$, the sum over all primes above $19$ in $\mathbb{Q}(\sqrt[3^n]{m})$, this tells us that $k$ is odd and there are an odd number of primes above $19$. Hence 
		\begin{equation}
			w(E/\mathbb{Q}(\sqrt[3^n]{m}))=(-1)^{\text{odd number}}\cdot (-1)^{n+1}=(-1)^{n},
		\end{equation}
		since $\mathbb{Q}(\sqrt[3^n]{m})$ has an even number of infinite places when $n$ is odd and an odd number when $n$ is even. Since the sign of the global root number changes at each step, the parity conjecture implies that the rank of $E$ must grow at every step of the tower $\big(\mathbb{Q}(\sqrt[3^n]{m})\big)_{n\geq 0}$. In particular, the rank of $E$ over $\mathbb{Q}(\sqrt[3^n]{m})$ is at least $n$. 
		
	\end{example}
	
	In \cite{vladtwists}, Example \ref{19ext} is generalised to the setting of root numbers of elliptic curves defined over $\mathbb{Q}$ in the towers of extensions $\big(\mathbb{Q}(\sqrt[p^n]{m})\big)_{n\geq 0}$ and $\big(\mathbb{Q}(\zeta_{p^n},\sqrt[p^n]{m})\big)_{n\geq 0}$, where $p$ is an odd prime. Here we present the results in the setting of semistable elliptic curves.
	
	\begin{theorem}[\cite{vladtwists} Theorem 6]\label{tower}
		Let $E$ be a semistable elliptic curve over $\mathbb{Q}$. Let $p$ be an odd prime at which $E$ has good reduction. Let $m>1$ be an $p$-th power free integer\footnote{We call $m$ `$p$-th power free' if $x^p\nmid m$ for any $x$, i.e. it is the generalisation of square-free and cube-free.}. Then the global root number for $E$ over $K=\mathbb{Q}(\sqrt[p^n]{m})$ is
		\begin{equation}
			w(E/K)=w(E/\mathbb{Q})\cdot (-1)^{n(\frac{p-1}{2}+t)},
		\end{equation}
		where $t$ is the number of primes of multiplicative reduction of $E$ that do not divide $m$ and that are non-squares modulo $p$.
	\end{theorem}
	
	\begin{corollary}
		Let $E$ be a semistable elliptic curve over $\mathbb{Q}$ and let $p\equiv 3\pmod{4}$ be a prime at which $E$ has good reduction. Suppose that every prime of multiplicative reduction is a square modulo $p$. If the parity conjecture holds then the rank of $E$ over $\mathbb{Q}(\sqrt[p^n]{m})$ is at least $n$ for every $p$-th power free integer $m$.
	\end{corollary}
	
	\begin{proof}
		Since $w(E/\mathbb{Q}(\sqrt[p^n]{m}))=w(E/\mathbb{Q})\cdot (-1)^n$, assuming the parity conjecture the rank of $E$ increases at every step of the tower of number fields $\big(\mathbb{Q}(\sqrt[p^n]{m})\big)_{n\geq 0}$. 
	\end{proof}
	
	\begin{example}\label{11}
		Assuming the parity conjecture, the curve
		\begin{equation}
			E:y^2+y=x^3-x^2 \hspace{10pt} \text{(11A3)}
		\end{equation}
		has rank at least $n$ over $\mathbb{Q}(\sqrt[7^n]{m})$ for any positive $7$-th power free integer $m$. Indeed, we have $\Delta_E=-11$ and $E$ has split multiplicative reduction at $11$. Since $11$ is a square modulo $7$, the rank is at least $n$ over $\mathbb{Q}(\sqrt[7^n]{m})$. In particular, the parity conjecture implies that $E$ has rank at least $1$ over $\mathbb{Q}(\sqrt[7]{3})$. Magma fails to find the point of infinite order on $E/\mathbb{Q}(\sqrt[7]{3})$.
	\end{example}
	
	\subsection{$E:y^2+y=x^3+x^2+x$ has rank at least $3^n$ over $\mathbb{Q}(\zeta_{3^n},\sqrt[3^n]{m})$}\label{towersection}
	
	\begin{example}\label{19zeta}
		Let us study what happens to the rank of the curve $E:y^2+y=x^3+x^2+x$ (19A3) from Examples \ref{19} and \ref{19ext} over $K_n=\mathbb{Q}(\zeta_{3^n},\sqrt[3^n]{m})$. We claim that, assuming the parity conjecture for twists (see \S\ref{paritytwists}), $E$ has rank at least $3^n$ over $\mathbb{Q}(\zeta_{3^n},\sqrt[3^n]{m})$, for every cube-free integer $m>1$. Let $G=\text{Gal}(K_n/\mathbb{Q})$ and $V=E(K_n)\otimes_{\mathbb{Z}}\mathbb{C}$. Then, by Lemma \ref{frob}, for any subgroup $H\leq \text{Gal}(K_n/\mathbb{Q})$, we have 
		\begin{equation}
			\text{rk}(E/K_n^H)=\langle \text{Ind}_{H}^{G}\mathds{1},V\rangle_G. 
		\end{equation}
		We can use this, along with our knowledge from Example \ref{19ext} about what happens to $E$ over $\mathbb{Q}(\sqrt[3^n]{m})$, to say something about the rank of $E$ over $K_n$. First, note that 
		\begin{equation}
			\sigma_n=\text{Ind}_{\text{Gal}(K_n/\mathbb{Q}(\sqrt[3^n]{m}))}^{G} \mathds{1} = \mathds{1}\oplus\rho_1\cdots\oplus \rho_n,
		\end{equation}
		where $\rho_k$ is an irreducible representation of dimension $p^{k-1}(p-1)$, see \cite{vladtwists} \S 5.2 for more details. By Theorem \ref{artinformalism}, we have
		\begin{equation}
			w(E/\mathbb{Q}(\sqrt[3^n]{m}))=w(E/\mathbb{Q})\cdot w(E/\mathbb{Q},\rho_1)\cdots w(E/\mathbb{Q},\rho_n),
		\end{equation}
		whence, $w(E/\mathbb{Q},\rho_k)=-1$ as we saw in Example \ref{19ext} that $w(E/\mathbb{Q})=+1$ and $w(E/\mathbb{Q}(\sqrt[3^n]{m}))=(-1)^n$. Assuming the parity conjecture for twists, this tells us that $\rho_k$ appears in $V$ with odd multiplicity for all $k\in\{1,\dots, n\}$. Now, $\mathbb{C}[G]$ has $\rm{dim} \ \rho_k$ copies of $\rho_k$, whereby $\text{rk}(E/K_n)=\langle\mathbb{C}[G],V\rangle_G$ is at least $\sum_{k=1}^n \text{dim} \rho_k=3^n-1$. Note that $w(E/\mathbb{Q}(\zeta_3))=-1$ since $19$ splits in $\mathbb{Q}(\zeta_3)$, so the parity conjecture implies $\text{rk}(E/\mathbb{Q}(\zeta_3))\geq 1$ and, in particular,
		\begin{equation}
			\text{rk}(E/\mathbb{Q}(\zeta_3))=\langle\mathbb{C}[{\rm{Gal}}(\mathbb{Q}(\zeta_3)/\mathbb{Q})],V\rangle_G\geq 1.
		\end{equation}
		Since $\mathbb{C}[\rm{Gal}(\mathbb{Q}(\zeta_3)/\mathbb{Q})]$ appears in $\mathbb{C}[G]$ and is orthogonal to the $\rho_k$, along with our previous working this tells us that 
		\begin{equation}
			\text{rk}(E/K_n)=\langle\mathbb{C}[G],V\rangle_G\geq3^n. 
		\end{equation}
	\end{example}
	
	\begin{remark}
		By Lemma \ref{permutation}, the result in the above example could have been deduced from the parity conjecture for $E$ over intermediate fields rather than the parity conjecture for twists.
	\end{remark}
	
	In \cite{vladtwists} the methods demonstrated in the example above are generalised in order to prove the following result (which for simplicity we state just for semistable elliptic curves).
	
	\begin{theorem}[\cite{vladtwists} Corollary 13]\label{rankpower}
		Let $p\equiv 3\pmod{4}$ be a prime number and let $E$ be a semistable elliptic curve with good reduction at $p$ and such that every prime of multiplicative reduction is a square modulo $p$. If the parity conjecture holds then, for every $p$-th power free integer $m>1$, the rank of $E$ over $\mathbb{Q}(\zeta_{p^n},\sqrt[p^n]{m})$ is at least $p^n$. 
	\end{theorem}
	
	\begin{example}
		Assuming the parity conjecture, the curve $E:y^2+y=x^3-x^2$ (11A3) from Example~\ref{11} has rank at least $7^n$ over $\mathbb{Q}(\zeta_{7^n},\sqrt[7^n]{m})$, for any $7$-th power free integer $m>1$.
	\end{example}
	
	\begin{remark}
		For a possible approach to explaining these points of infinite order using Heegner points, see \cite{darmon}.
	\end{remark}
	
	\subsection{Every positive $d\in\mathbb{Q}^\times/\mathbb{Q}^{\times 2}$ can be represented by $x^3-91x-182$}
	Theorem \ref{5050quad} tells us that half of the quadratic twists of an elliptic curve over $\mathbb{Q}$ by $d$ have global root number $+1$ and half have global root number $-1$. This distribution can be such that the root number is determined only by the sign of $d$, and a description of elliptic curves $E/\mathbb{Q}$ for which this occurs can be found in \cite{desjardinstwists} Theorem 1.1. If we believe the parity conjecture, we can use this to deduce that every positive rational number $d$ can be written in the form $d=s^2(t^3-91t-182)$ for $s$ and $t$ in $\mathbb{Q}$. For example, letting $f(x)=x^3-91x-182$:
	\begin{equation}
		1 =\frac{1}{8^2}\cdot f(-3), \hspace{15pt} 2 =\frac{27^2}{16^2}\cdot  f\left(-\frac{19}{9}\right), \hspace{15pt} 3 =\frac{9^2}{64^2}\cdot f\left(-\frac{17}{3}\right), \hspace{15pt} 5= \frac{675^2}{1185848^2}\cdot f\left(\frac{11209}{45}\right). 
	\end{equation}
	It is not clear how one might prove this property of $f(x)$ without the parity conjecture. 
	
	\begin{example}[As in \cite{desjardinstwists} Example 4]\label{posquad}
		We claim that all positive quadratic twists of the curve
		\begin{equation}
			E:y^2 = x^3-91x+182 \hspace{10pt}\text{(8281H1)}
		\end{equation}
		have root number $+1$ and all negative quadratic twists have root number $-1$. The minimal discriminant of $E$ is $\Delta=7^2\cdot 13^2$ and $E$ has type $\text{II}$ reduction at both primes. By Theorem \ref{rootnos}, $w(E/\mathbb{Q}_7)=(-1)^{\lfloor\frac{2\cdot 7}{12} \rfloor}=-1$ and similarly $w(E/\mathbb{Q}_{13})=+1$. Hence $w(E/\mathbb{Q})=+1$. Let $d\in\mathbb{Q}^\times\setminus\mathbb{Q}^{\times 2}$, $K=\mathbb{Q}(\sqrt{d})$ and let $v$ denote a prime above $7$ in $K$. If $7$ is ramified in $K$ then 
		\begin{equation}
			w(E/K_v)=(-1)^{\lfloor \frac{4\cdot 7}{12}\rfloor}=+1. 
		\end{equation}
		By a similar calculation, if $7$ is inert in $K$ then $w(E/K_v)=+1$. If $7$ splits, by Lemma \ref{2primes} there are two distinct places with the same root number. In all cases $\prod_{v|7}w(E/K_v)=+1$. We can do the same calculation for $\Tilde{v}$ above $13$ to find $\prod_{\Tilde{v}|13}w(E/K_{\Tilde{v}})=+1$. Thus 
		\begin{equation}
			w(E/K)=\begin{cases}
				+1 &\text{ if } d>0; \\
				-1 &\text{ if } d<0, 
			\end{cases}
		\end{equation}
		and our claim follows from Lemma \ref{quadrootnos}. In particular, the parity conjecture implies that all negative quadratic twists of $E$ have infinitely many points. 
	\end{example}
	
	\begin{definition}
		We say that $d\in\mathbb{Q}^\times/\mathbb{Q}^{\times 2}$ can be represented by $f(x)\in\mathbb{Q}[x]$ if there exists an $s\in\mathbb{Q}^\times$ and $t\in \mathbb{Q}$ such that $d=s^2f(x)$. If $f(x)$ is a cubic and the quadratic twist of the elliptic curve $E:y^2=f(x)$ by $d\in\mathbb{Q}^\times/\mathbb{Q}^{\times 2}$ has positive rank, then $d$ is represented by $f(x)$. 
	\end{definition}

	\begin{example}\label{negquad}
		Let 
		\begin{equation}
			E:y^2 = x^3-91x-182. \hspace{10pt} \text{(132496E1)}
		\end{equation}
		This is the quadratic twist by $-1$ of the curve in Example \ref{posquad} so
		\begin{align}
			w(E_d/\mathbb{Q})=\begin{cases}
				-1 &\text{ for } d>0; \\
				+1 &\text{ for } d<0. 
			\end{cases}
		\end{align}
		The parity conjecture implies that every positive $d\in\mathbb{Q}^{\times}/\mathbb{Q}^{\times 2}$ is representable by $x^3-91x-182$. 
	\end{example}
	
	\begin{remark}
		Assuming Goldfeld's conjecture that there is a $50/50$ distribution of the rank of quadratic twists of an elliptic curve $E/\mathbb{Q}$ being $0$ and $1$, from Example \ref{negquad} we deduce\footnote{None of the quadratic twists of $E$ have torsion points as Galois has maximal image on $E[p]$ for all $p>2$ and $E(\mathbb{Q})[2]=0$.} that $0\%$ of negative integers (up to squares) can be represented by $x^3-91x-182$. This is a peculiar property for a cubic polynomial: its positive values hit all possible classes modulo squares whereas its negative values hit a very sparse set.
	\end{remark}
	
	\subsection{Every $d\in\mathbb{Q}^\times/\mathbb{Q}^{\times 2}$ can be represented by $4x^3 -32x - 35$\hspace{5pt} or \hspace{5pt}$9x^3+16x+16$}\label{represented}
	Example \ref{negquad} and Example \ref{posquad} show that, assuming the parity conjecture, for every $d\in\mathbb{Q}$ at least one of the equations
	\begin{equation}
		d\cdot y^2=x^3-91x+182 \hspace{30pt} \text{and} \hspace{30pt} d\cdot y^2=x^3-91x-182
	\end{equation}
	has a rational solution. In particular, every $d\in\mathbb{Q}^\times/\mathbb{Q}^{\times 2}$ can be represented by $x^3-91x+182$ or $x^3-91x-182$. It turns out that every cubic $f(x)$ has an auxiliary cubic $g(x)$ with this property: 
	
	\begin{theorem}
		Let $f(x)\in\mathbb{Q}[x]$ be a separable cubic polynomial. Assuming the parity conjecture, there exists a separable cubic polynomial $g(x)\in\mathbb{Q}[x]$ such that every $d\in\mathbb{Q}^\times/\mathbb{Q}^{\times 2}$ is represented by $f(x)$ or $g(x)$.
	\end{theorem}
	
	\begin{proof}
		Take $g(x)=d_0f(x)$, where $d_0<0$ is such that all primes of bad reduction of $E:y^2=f(x)$ split in $\mathbb{Q}(\sqrt{d_0})$.  Then $E_{d_0}:y^2=g(x)=d_0f(x)$ is the quadratic twist of $E$ by $d_0$. By the proof of Theorem \ref{5050quad}(ii), for any $d\in\mathbb{Q}^\times/\mathbb{Q}^{\times 2}$ the root number of the quadratic twists of $E_{d_0}$ and $E$ by $d$ satisfy $w(E_{dd_0}/\mathbb{Q})=-w(E_{d}/\mathbb{Q})$. The parity conjecture implies that at least one of $E_d$ and $E_{dd_0}$ has a point of infinite order, and thus every $d\in\mathbb{Q}^\times/\mathbb{Q}^{\times 2}$ is represented by $f(x)$ or $g(x)$. 
	\end{proof}
	
	We can also find examples of such curves that are not quadratic twists of each other.
	
	\begin{example}
		Let
		\begin{equation}
			E:y^2 = x^3 + 16x^2 - 3072x - 68608.
		\end{equation}
		By similar calculations to those in Example \ref{posquad}, all the positive quadratic twists of the curve have root number $+1$ over $\mathbb{Q}$, and all the negative quadratic twists have root number $-1$. Thus, every $d\in\mathbb{Q}^\times/\mathbb{Q}^{\times 2}$ can be represented by $x^3-91x-182$ or by $x^3 + 16x^2 - 3072x - 68608$. 
	\end{example}
	
	Theorem \ref{5050quad} tells us that the global root number in a quadratic twist family is periodic; we can use this to come up with curves that are not quadratic twists of each other, for which one of their quadratic twists always has root number $-1$ and where the root number does not only depend on the sign of $d$.
	
	\begin{example}\label{largehight}
		Let $E$ and $E'$ be given by 
		\begin{equation}
			E:y^2 = 4x^3 -32x - 35 \quad \text{and}\quad E':y^2 = 9x^3+16x+16. 
		\end{equation}
		We claim that for every $d\in \mathbb{Q}^\times/\mathbb{Q}^{\times 2}$, either $w(E_d/\mathbb{Q})=-1$ or $w(E'_{d}/\mathbb{Q})=-1$.
		
		Both $E$ and the quadratic twist $E_{-3}'$ have minimal discriminant $-307$ and split multiplicative reduction at $307$. Thus $w(E/\mathbb{Q})=w(E_{-3}'/\mathbb{Q})=+1$. Since $307\equiv 1\pmod{3}$, $\mathbb{Q}_{307}$ contains $\sqrt{-3}$ and so $307$ has the same splitting behaviour in $\mathbb{Q}(\sqrt{d})$ as it does in $\mathbb{Q}(\sqrt{-3d})$. Thus $w(E/\mathbb{Q}(\sqrt{d}))=-w(E_{-3}'/\mathbb{Q}(\sqrt{-3d}))$ by Corollary \ref{semistable} and, by Lemma \ref{quadrootnos}, $w(E_d/\mathbb{Q})=-w(E'_{d}/\mathbb{Q})$. This proves the claim. The parity conjecture implies that every $d\in\mathbb{Q}^\times/\mathbb{Q}^{\times 2}$ can be represented as $f(x)=4x^3 -32x - 35$ or $g(x)=9x^3+16x+16$. 
	\end{example}
	
	\begin{remark}
		This representation of $d\in\mathbb{Q}^\times/\mathbb{Q}^{\times 2}$ as one of two cubics can be given by values of $x$ with very large height. In Example \ref{largehight}, the `smallest' (in terms of the height of $x$) representation of $-17$ is 
		\begin{equation}
			-17  = \frac{223280590408502625944275649356615800897^2}{6034892905593758120664851652128224164489^2}\cdot g\left(-\frac{330029825965445476649347569}{29773633899834719965546857}\right). 
		\end{equation}
	\end{remark}
	
	\subsection{All quadratic twists of $y^2=x^3+\frac{5}{4}x^2-2x-7$ over $\mathbb{Q}(\zeta_3,\sqrt[3]{11})$ have infinitely many points}\label{nogoldfeld5}
	As mentioned at the end of \S\ref{rootnosquadtwists}, the $50/50$ distribution of root numbers of quadratic twists of elliptic curves over $\mathbb{Q}$ can fail quite dramatically over number fields; there are elliptic curves over number fields for which all quadratic twists have the same root number.
	
	\begin{example}[As considered in \cite{posquadtwists}]\label{nogoldfeld3}
		Let $K=\mathbb{Q}(\zeta_3,\sqrt[3]{11})$. We claim that all quadratic twists of the curve 
		\begin{equation}
			E:y^2=x^3+\frac{5}{4}x^2-2x-7
		\end{equation}
		over $K$ have root number $-1$ and so should have positive rank. Indeed, the minimal discriminant of $E/\mathbb{Q}$ is $11^4$, and $E$ acquires everywhere good reduction over $K$. Since $K$ has three complex places and $K(\sqrt{d})$ has six, for $d\in K^{\times}\setminus K^{\times 2}$, we have $w(E/K)=-1$ and $w(E/K(\sqrt{d}))=+1$. By Lemma \ref{quadrootnos}
		\begin{equation}
			w(E_d/K)=w(E/K(\sqrt{d}))w(E/K)=-1.
		\end{equation}
		So, assuming the parity conjecture, $E_d$ has a point of infinite order for every $d\in K^\times$. See Theorem \ref{fakecm} below for a more general result concerning this behaviour.
	\end{example}
	
	\begin{remark}
		Example \ref{nogoldfeld3} shows that, assuming the parity conjecture, the polynomial $x^3+5/4x^2-2x-7$ takes every value in $K^\times/K^{\times2}$. 
	\end{remark}

	\subsection{$E:y^2=x^3+x^2-12x-\frac{67}{4}$ has even rank over every field extension of $\mathbb{Q}(\sqrt[4]{-37})$}
	
	Consider the elliptic curve
	\begin{equation}
		E:y^2=x^3+x.
	\end{equation}
	It has complex multiplication defined over $K=\mathbb{Q}(i)$ with $\text{End}_K(E)\cong \mathbb{Z}[i]$, where $i$ acts by $[i]\cdot(x,y)=(-x,iy)$. The rank of $E$ over $K$ is even, and similarly $\text{rk}(E/F)$ is even for any extension $F$ of $K$, since $E(F)\otimes_{\mathbb{Z}}\mathbb{Q}$ is naturally a $\mathbb{Q}(i)$-vector space, and so has even dimension over $\mathbb{Q}$. Moreover, using the fact that $\text{rk}(E_d/K)=\text{rk}(E/K(\sqrt{d}))-\text{rk}(E/K)$, we see that $\text{rk}(E_d/K)$ is even for any $d\in K^\times$. The parity conjecture forces some non-CM elliptic curves to exhibit the same behaviour. 
	
	\begin{example}[As considered in \cite{posquadtwists}]\label{nogoldfeld1}
		Let us take the curve 
		\begin{equation}
			E:y^2=x^3+x^2-12x-\frac{67}{4}. \hspace{10pt}\text{(1369E1)}
		\end{equation}
		Then $\Delta_E=37^3$ and $j(E)=2^{12}$, so $E/\mathbb{Q}$ has potentially good reduction at $37$ and it has everywhere good reduction over $K=\mathbb{Q}(\sqrt[4]{-37})$ and over every extension $F$ of $K$. Since $F$ has an even number of infinite places, $w(E/F)=+1$. In particular, assuming the parity conjecture, $E$ has even rank over every extension of $K$. 
	\end{example} 
	
	\begin{remark}
		Such a field $K$ exists for every elliptic curve $E$ with integral $j$-invariant as this means that $E$ has potentially good reduction everywhere, so we can find an extension of the base field over which $E$ acquires everywhere good reduction.
	\end{remark}
	
	As in \S\ref{parityex} and \S\ref{c2c2}, one would hope that there is some extra structure analogous to complex multiplication that would explain why $E$ always has even rank over extensions of $K$ in Example \ref{nogoldfeld1}. However, it is not clear how to prove this behaviour of ranks for any non-CM curve without assuming the parity conjecture. Elliptic curves that display this phenomenon, and that described in \S\ref{nogoldfeld5}, are classified by the following theorem. 
	
	\begin{theorem}[\cite{posquadtwists} Theorem 1]\label{fakecm}
		For an elliptic curve $E$ over a number field $K$, the following are equivalent
		\begin{enumerate}[(i)]
			\item All quadratic twists of $E/K$ have the same root number;
			\item $w(E/F)=w(E/K)^{[F:K]}$ for every finite extension $F$ of $K$;
			\item $K$ has no real places, and $E$ aquires everywhere good reduction over an abelian extension of $K$. 
		\end{enumerate}
	\end{theorem}
	
	Theorem \ref{fakecm} tells us that if $E/K$ satisfies $(iii)$ and $w(E/K)=+1$ then, assuming the parity conjecture, $E$ will have even rank over every extension of $K$. 
	
	\begin{remark}
		It turns out that the criteria in Theorem \ref{fakecm} are equivalent to $K$ having no real places, and for all primes $\ell$ and all places $v\nmid \ell$ of $K$, the action of the absolute Galois group of $K_v$ on the Tate module $T_\ell(E)$ being abelian (see \cite{posquadtwists}). This is another way in which $E$ resembles a CM curve, since (in view of the Tate conjecture) an elliptic curve has CM if and only if the action of the global absolute Galois group on $\ell$-adic Tate module is abelian.
	\end{remark}
	
	\subsection{The rank of $E:y^2+xy=x^3-x^2-2x-1$ grows in extensions of $\mathbb{Q}(\sqrt{-1})$ of even degree} 
	
	Theorem \ref{fakecm} tells us that there exist elliptic curves over number fields for which all quadratic twists have the same root number. We can also use the theorem to come up with examples of elliptic curves $E/K$ whose rank grows in every even degree extension of $K$.
	
	\begin{example}[As considered in \cite{posquadtwists}]\label{nogoldfeld2}
		Let 
		\begin{equation}
			E:y^2+xy=x^3-x^2-2x-1. \hspace{10pt} \text{(49A1)}
		\end{equation}
		We have $\Delta_E=-7^3$, $j(E)=-3^3\cdot 5^3$ and $E$ has additive, potentially good, type III reduction at $7$. We claim that the rank of $E$ must grow in every extension of $K=\mathbb{Q}(\sqrt{-1})$ of even degree. Indeed, $7$ is inert in $\mathbb{Q}(\sqrt{-1})$ and by Theorem \ref{rootnos}, $w(E/K_v)=+1$, where $v$ is the unique prime above $7$ in $K$, and so $w(E/K)=-1$. Note that $K$ has no real places and $E$ acquires everywhere good reduction over $\mathbb{Q}(\sqrt{-1},\sqrt[4]{7})$, which is an abelian extension of $K$. By Theorem \ref{fakecm}, for any even degree extension $F$ of $K$, 
		\begin{equation}
			w(E/F)=w(E/K)^{[F:K]}=(-1)^{\text{even number}}=+1\neq w(E/K),
		\end{equation}
		and the parity conjecture predicts that $\text{rk}(E/F)>\text{rk}(E/K)$.
	\end{example}
	
	\subsection{Every positive integer $n\equiv 5, 6, \text{ or } 7  \pmod{8}$ is a congruent number}\label{congnumberprob}
	
	The classical \textit{congruent number problem} asks: for a natural number $n$, can it be realised as the area of a right-angled triangle with rational sides? We call such $n$ \textit{congruent numbers}. The congruent number problem dates back to Arab manuscripts from the 10th century and to this day it is not known in full generality, although it has been extensively studied. Heegner \cite{heegner} proved that if $p$ is a prime congruent to $5$ modulo $8$ then it is a congruent number, and other cases have been proved using the theory of Heegner points on modular curves, see \cite{tian}. Smith \cite{smith1} has announced a proof that at least $55.9\%$ of positive square-free integers congruent to $5$, $6$, or $7$ modulo $8$ are congruent numbers. It has been long expected that every positive integer congruent to $5, 6, \text{ or } 7  \pmod{8}$ is a congruent number, and we will explain how this can be deduced from the parity conjecture.
	
	\begin{example}
		There is a right-angled triangle with sides of length $3$, $4$, $5$ which has area $6$, so that $6$ is a congruent number. Similarly, the right-angled triangle with sides of length $\frac{35}{12}$, $\frac{24}{5}$ and $\frac{337}{60}$ has area $7$, whence $7$ is a congruent number.
	\end{example}

	\begin{definition}
		The elliptic curve $E:y^2=x^3-x$ is called the \textit{congruent number curve} because the quadratic twist of the curve by $n\in\mathbb{N}$, $E_n:y^2=x(x-n)(x+n)$, has a point of infinite order over $\mathbb{Q}$ if and only if $n$ is a congruent number (see e.g. \cite{tunnell} for details on why this is the case). For these elliptic curves, a rational point $(x,y)$ has infinite order if and only if $y\neq 0$. 
	\end{definition}
	
	\begin{example}
		We already know that $7$ is a congruent number. Indeed, the elliptic curve
		\begin{equation}
			E_7:y^2=x(x-7)(x+7) 
		\end{equation}
		has a rational point of infinite order, namely $(25,120)$. 
	\end{example}
	
	\begin{example}
		There is a point on the curve
		\begin{equation}
			E_{166}:y^2= x(x-166)(x+166)
		\end{equation}
		which is given by
		\begin{align}
			x= -\frac{7969693283}{98823481} \text{ and } y=-\frac{1280060076599271}{982404224621}.
		\end{align}
		This tells us that $166$ is a congruent number. The height of the generator $E_{166}$ is already very large. What if we wanted to know whether $800\hspace{3pt}006$ was a congruent number? Magma \cite{magma} cannot find a point of infinite order on $E_{800006}$ and cannot tell us whether $800\hspace{3pt}006$ is not a congruent number. Assuming the parity conjecture, we can easily show that it is.
	\end{example}
	
	\begin{theorem}\label{congnum}
		Let $n$ be a positive, square-free integer. Then 
		\begin{equation}
			w(E_n/\mathbb{Q})=
			\begin{cases}
				+1 & \text{ if } n\equiv 1,\text{ } 2 \text{ or } 3 \pmod{8}, \\
				-1 & \text{ if } n\equiv 5,\text{ } 6 \text{ or } 7 \pmod{8}. 
			\end{cases}
		\end{equation}
		
	\end{theorem}
	
	\begin{proof}
		By Theorem \ref{5050quad}$(i)$, $w(E_n/\mathbb{Q})$ depends only on $n\pmod{16}$. So all one needs to do is calculate the root numbers for one square-free $n$ in each of the congruence classes modulo $16$. For $E_1:y^2=x^3-x$, we have $c_4=2^4\cdot 3$, $c_6=0$ and $\Delta_{E_1}=2^6$. In the terminology of Notation \ref{cprime}, we have $c_4'=3$ and $c_4'-4c_{6,7}=3$ so that, by the table in Appendix \ref{rootnosat2table}, $w(E_1/\mathbb{Q}_2)=-1$ and hence $w(E_1/\mathbb{Q})=+1$. For $E_5:y^2=x^3-25x$, we have $c_4=2^4\cdot 3\cdot 5^2$, $c_6=0$, $\Delta_{E_5}=2^6\cdot 5^6$, $c_4'=75\equiv 3\pmod{4}$ and $c_4'-4c_{6,7}\equiv 11\pmod{16}$. Again by the table, $w(E_5/\mathbb{Q}_2)=+1$. By Theorem \ref{rootnos}, $w(E_5/\mathbb{Q}_5)=(-1)^{\lfloor\frac{6\cdot 5}{12}\rfloor}=+1$, so $w(E_5/\mathbb{Q})=-1$. Similarly, we have 
		\begin{equation}
			w(E_{41}/\mathbb{Q})= w(E_2/\mathbb{Q})=w(E_{10}/\mathbb{Q})= w(E_3/\mathbb{Q})=w(E_{11}/\mathbb{Q})=+1;
		\end{equation}
		\begin{equation}
			w(E_{13}/\mathbb{Q})= w(E_6/\mathbb{Q})=w(E_{14}/\mathbb{Q})= w(E_7/\mathbb{Q})=w(E_{15}/\mathbb{Q})=-1,
		\end{equation}
		which completes the proof. 
	\end{proof}

	\begin{corollary}
		Assuming the parity conjecture, every square-free positive integer $n\equiv 5, 6, \text{ or } 7  \pmod{8}$ is the area of a right-angled triangle with rational sides. 
	\end{corollary}
	
	We can deduce immediately from this corollary that, assuming the parity conjecture, $800\hspace{3pt}006$ is a congruent number. Stephens was the first to apply parity-type conjectures to the congruent number problem \cite{stephens} and proved that Selmer's second conjecture (see \S \ref{selmerstuff}) implies that positive integers $n\equiv 5,6,7 \pmod{8}$ are congruent numbers.
	
	\subsection{$\mathscr{E}:y^2=x(x^2-49(1+t^4)^2)$ has a point of infinite order for every $t\in\mathbb{Q}$ but ${\rm{rk }}\text{ }\mathscr{E}/\mathbb{Q}(t)=0$}\label{familysection}
	
	One incentive for studying root numbers is to study ranks in families of elliptic curves. Consider a one parameter family of elliptic curves
	\begin{equation}
		\mathscr{E}:y^2=x^3+A(t)x+B(t),
	\end{equation}
	for some fixed $A(t)$, $B(t)\in \mathbb{Q}(t)$, where $\Delta_{\mathscr{E}}\neq 0$. A point of infinite order over $\mathbb{Q}(t)$ gives a parametric family of points over the fibres $\mathscr{E}_t$. A natural question to ask is `if every fibre has positive rank, does it mean there is a point of infinite order over $\mathbb{Q}(t)?$' Assuming the parity conjecture, the answer is `no'.

	\begin{theorem}[As considered in \cite{cassels3}]\label{cassels}
		Let $\mathscr{E}:y^2=x(x^2-49(1+t^4)^2)$. Then
		\begin{enumerate}[(i)]
			\item ${\rm{rk}}\text{ } \mathscr{E}/\mathbb{Q}(t)=0$;
			\item Assuming the parity conjecture, for every $t\in\mathbb{Q}$ the fibre $\mathscr{E}_t$ has $\textup{rk}(\mathscr{E}_t/\mathbb{Q})\geq 1$.
		\end{enumerate}
	\end{theorem}
	
	\begin{proof}
		In \cite{cassels3}, it is shown that $\Tilde{\mathscr{E}}:y^2=x(x^2-(1+t^4)^2)$ has rank $0$ over $\mathbb{C}(t)$. We get $(i)$ from the fact that this curve is isomorphic to $\mathscr{E}$ over $\mathbb{C}(t)$, and so $\rm{rk} \text{ }\mathscr{E}/\mathbb{Q}(t)=0$. For $(ii)$, let $t=l/m$ where $(l,m)=1$. Then $\mathscr{E}_t$ is isomorphic to a curve of the form
		\begin{equation}
			y^2=x(x-n)(x+n)
		\end{equation}
		over $\mathbb{Q}$, where $n=7l^4+7m^4$. Clearly, $n\equiv 6\text{ or } 7\pmod{8}$ so, by Theorem \ref{congnum}, $w(\mathscr{E}_t/\mathbb{Q})=-1$.
	\end{proof} 
	
	The family $\mathscr{E}$ in Theorem \ref{cassels} has constant $j$-invariant. In \cite{desjardins} Theorem 1.2, Desjardins proved that, for a family with non-constant $j$-invariant, assuming Chowla's conjecture and the Squarefree conjecture (see \cite{desjardins} Conjectures 2.1 and 2.7), the function $t\mapsto w(\mathscr{E}_t/\mathbb{Q})$ is never constant. She used this to prove that the sets 
	\begin{align}
		W_+&=\{t\in\mathbb{P}^1:w(\mathscr{E}_t/\mathbb{Q})=+1\}, \\
		W_-&=\{t\in\mathbb{P}^1:w(\mathscr{E}_t/\mathbb{Q})=-1\}
	\end{align}
	are infinite and so, assuming the parity conjecture, the rational points $\mathscr{E}(\mathbb{Q})$ are Zariski dense in $\mathscr{E}$. 
	
	Since the sets $W_+$ and $W_-$ are infinite when the $j$-invariant is non-constant, one might expect that the root numbers are equidistributed and that the average root number is $0$. This is not always the case, as shown in \cite{helfgott}, where Helfgott proved that if $\mathscr{E}$ has no places of multiplicative reduction, the average root number need not be zero.

	For a family with non-constant $j$-invariant, if one restricts to $t\in\mathbb{Z}$ the function $t\mapsto w(\mathscr{E}_t/\mathbb{Q})$ may be constant. For example, in \cite{rizzo} it is shown that the surface 
	\begin{equation}
		\mathscr{E}:y^2= x^3+ tx^2-(t+3)x+1 
	\end{equation}
	has $w(\mathscr{E}_t/\mathbb{Q})=-1$ for every $t\in\mathbb{Z}$. This example was first investigated by Washington \cite{washington} who proved via $2$-descent that, for every $t\in\mathbb{Z}$ such that $t^2+3t+9$ is square free and assuming the finiteness of the Tate–Shafarevich group, the rank of $\mathscr{E}_t$ is odd.
	
	\subsection{$y^2+y=x^3-x$ acquires new solutions over $K(\sqrt{\Delta_K})$ whenever $37\nmid \Delta_K$ and $\sqrt{\Delta_K}\not\in K$}\label{sqrtK}
	
	We saw in \S\ref{firsttower} and \S\ref{towersection} that the parity conjecture can predict the growth of the rank of elliptic curves in towers of fields of the form $\mathbb{Q}(\sqrt[p^n]{m})$ and $\mathbb{Q}(\zeta_{p^n},\sqrt[p^n]{m})$. The parity conjecture also predicts that the rank of an elliptic curve over $\mathbb{Q}$ can grow in many generic number fields.
	
	\begin{theorem}\label{sqrtdeltak}
		Let $K$ be any number field that does not contain $\sqrt{\Delta_K}$ and set $F=K(\sqrt{\Delta_K})$. Assuming the parity conjecture, every $E/\mathbb{Q}$ with $w(E/\mathbb{Q})=-1$ and $(N_E, \Delta_K)=1$ has $\textup{rk} (E/F) > \textup{rk} (E/\mathbb{Q})$. 
	\end{theorem}
	
	\begin{proof}
		Let $\Tilde{K}/\mathbb{Q}$ be the Galois closure of $K/\mathbb{Q}$, $G=\text{Gal}(\Tilde{K}/\mathbb{Q})$ and $H=\text{Gal}(\Tilde{K}/K)$.
		Write $\text{Ind}_H^G \mathds{1} = \mathds{1}\oplus\rho$ and $\epsilon$ for the nontrivial character of $\text{Gal}(\mathbb{Q}(\sqrt{\Delta_K})/\mathbb{Q})$. By Frobenius reciprocity (see Lemma \ref{frob}),
		\begin{equation}
			\text{rk}(E/F)=\langle \text{Ind}_{\text{Gal}(\Tilde{K}/F)}^{G}\mathds{1}, E(\Tilde{K})\otimes_\mathbb{Z}\mathbb{C}\rangle . 
		\end{equation}
		Note that
		\begin{equation}
			\text{Ind}_{\text{Gal}(\Tilde{K}/F)}^{G}\mathds{1} = \text{Ind}_H^G(\mathds{1}\oplus\text{Res}_H^G\epsilon) = \mathds{1}\oplus\rho\oplus\epsilon\oplus(\epsilon\otimes\rho)
		\end{equation}
		so, assuming the parity conjecture for twists, it will suffice to show that one of $w(E/\mathbb{Q},\rho)$, $w(E/\mathbb{Q},\epsilon)$ and $w(E/\mathbb{Q},\epsilon\otimes\rho)$ is $-1$. In fact, by Lemma \ref{permutation}, it suffices to assume the parity conjecture for $E$ over the subfields of $\Tilde{K}$.
		
		Let $K=\mathbb{Q}(\alpha)$. The permutation action of $G$ on the cosets of $H$ agrees with the permutation action on the Galois conjugates of $\alpha$, since both actions are transitive with $H$ as a point-stabiliser. Even permutations are precisely those fixing $\sqrt{\Delta_K}$, and therefore $\det(\text{Ind}_H^G\mathds{1})=\det\rho$ cuts out the extension $\mathbb{Q}(\sqrt{\Delta_K})/\mathbb{Q}$ so that $\det\rho=\epsilon$. It follows that $\det(\epsilon\otimes\rho)=\epsilon^{\otimes\dim\rho}\otimes\det\rho$ is either $\mathds{1}$ or $\epsilon$ if $\dim\rho$ is odd or even, respectively. By Theorem \ref{rootnostwists}, if $[K:\mathbb{Q}]$ is even then $\dim\rho$ is odd and $w(E/\mathbb{Q},\epsilon\otimes\rho) = w(E/\mathbb{Q})^{\dim\rho}=-1$. Similarly, if $[K:\mathbb{Q}]$ is odd then $\dim\rho$ is even and $w(E/\mathbb{Q},\rho)w(E/\mathbb{Q},\epsilon)=w(E/\mathbb{Q},\rho\oplus\epsilon)=w(E/\mathbb{Q})^{\dim\rho+1} =-1$.
	\end{proof}

	\begin{remark}
		Theorem \ref{sqrtdeltak} applies to generic extensions $K/\mathbb{Q}$. If $[K:\mathbb{Q}]=n>2$ and the Galois closure of $K$ over $\mathbb{Q}$ has Galois group $S_n$ then $K$ does not contain $\sqrt{\Delta_K}$. We therefore expect the rank of every $E/\mathbb{Q}$ with $w(E/\mathbb{Q})=-1$ to increase in splitting fields of most polynomials.
	\end{remark}

	\subsection{$E:y^2+y=x^3-x$ has rank at least $434$ over any $S_{14}$-extension $K/\mathbb{Q}$ with $37\nmid\Delta_K$}\label{snsection}
	
	\begin{example}\label{s5}
		Let $E/\mathbb{Q}$ be an elliptic curve with $w(E/\mathbb{Q})=-1$ and let $F$ be an $S_5$-extension of $\mathbb{Q}$ with $(N_E,\Delta_F)=1$. We claim that $\text{rk}(E/F)\geq 6$. Indeed, $S_5$ has seven irreducible representations, all of which are self-dual. These are the trivial representation $\mathds{1}$, sign representation $\epsilon$, two $4$-dimensional representations $\rho_1$ and $\rho_2=\rho_1\otimes\epsilon$, two $5$-dimensional representations $\tau_1$ and $\tau_2=\tau_1\otimes\epsilon$ and one $6$-dimensional representation $\sigma$. We can write
		\begin{equation}
			E(F)\otimes_\mathbb{Z}\mathbb{C}=\mathds{1}^{\oplus a}\oplus\epsilon^{\oplus b}\oplus\rho_1^{\oplus c}\oplus\rho_2^{\oplus d}\oplus\tau_1^{\oplus e}\oplus\tau_2^{\oplus f}\oplus\sigma^{\oplus g}.
		\end{equation}
		Since $\text{rk}(E/F)=\text{dim}(E(F)\otimes_\mathbb{Z}\mathbb{C})$, 
		\begin{equation}
			\text{rk}(E/F)=a+b+4c+4d+5e+5f+6g. 
		\end{equation}
		Now, $\det(\tau_1\otimes\epsilon)=\epsilon^{\otimes\dim\tau_1}\otimes\det\tau_1$, so either $\text{det}(\tau_1)=\mathds{1}$ or $\text{det}(\tau_2)=\mathds{1}$. By Theorem \ref{rootnostwists},
		\begin{equation}
			w(E/\mathbb{Q},\tau_1)=-1 \quad \text{or}\quad  w(E/\mathbb{Q},\tau_2)=-1.
		\end{equation}
		By assumption, we have $w(E/\mathbb{Q},\mathds{1})=w(E/\mathbb{Q})=-1$ so, assuming the parity conjecture for twists, $a$ is odd and either $e$ or $f$ is odd. Hence $\text{rk}(E/F)\geq 6$.
	\end{example}
	
	\pagebreak
	
	\begin{theorem}\label{order2char}
		Let $E/\mathbb{Q}$ be an elliptic curve with $w(E/\mathbb{Q})=-1$. Assume the parity conjecture for twists.
		\begin{enumerate}[(i)]
			\item  Let $\rho$ be an odd-dimensional irreducible self-dual Artin representation with $\det\rho=\mathds{1}$ and whose conductor is coprime to $N_E$. Then $\langle\rho,E(F)\otimes_\mathbb{Z}\mathbb{C}\rangle> 0$, where $\rho$ factors through $F/\mathbb{Q}$.
			\item For a Galois extension $F/\mathbb{Q}$ with $(N_E,\Delta_F)=1$,
			\begin{equation}
				\textup{rk}(E/F)\geq\frac{k}{m},
			\end{equation}
			where $k$ is the sum of dimensions of the odd-dimensional irreducible self-dual representations of $\textup{Gal}(F/\mathbb{Q})$ and $m$ is the number of one-dimensional representations of $\textup{Gal}(F/\mathbb{Q})$ of order $1$ or $2$.
		\end{enumerate}
		
	\end{theorem}
	
	\begin{proof}
		$(i)$ This follows immediately from Theorem \ref{rootnostwists}. $(ii)$ Let $G=\textup{Gal}(F/\mathbb{Q})$ and let $\rho$ be any irreducible odd-dimensional self-dual representation of $G$. Denote the order $2$ one-dimensional representations of $G$ by $\epsilon_1,\dots,\epsilon_{m-1}$. Since $\rho$ is odd-dimensional, taking $\epsilon=\det\rho$ we obtain $\det(\rho\otimes\epsilon)=\epsilon^{\otimes\dim\rho}\otimes\det\rho=\mathds{1}$. We also have $\det(\rho\otimes\epsilon\otimes\epsilon_i)=\epsilon_i$ for $i=1,\dots,m-1$. Thus, the odd-dimensional irreducible self-dual representations come in sets of size $m$, each with the same dimension and different determinant character $\mathds{1},\epsilon_1,\dots,\epsilon_{m-1}$. By $(i)$ $\rho\otimes\epsilon$ appears in $E(F)\otimes_\mathbb{Z}\mathbb{C}$ and the result follows.
	\end{proof}
	
	\begin{remark}
		Theorem \ref{order2char} implies that for a fixed elliptic curve $E/\mathbb{Q}$ with $w(E/\mathbb{Q})=-1$, every odd-dimensional irreducible self-dual Artin representation which has trivial determinant and whose conductor is coprime to $N_E$ must appear in the Mordell--Weil group of $E$ over an extension of $\mathbb{Q}$.
	\end{remark}
	
	To generalise Example \ref{s5} to finding a lower bound on the rank of $E$ over an $S_n$-extension of $\mathbb{Q}$, we need to understand the odd-dimensional irreducible representations of $S_n$. For a given $n$, the dimensions of the irreducible representations of $S_n$ can be calculated using the Young tableaux, see \cite{fulton} p. 50. We will instead use a known case of the McKay conjecture to find a crude lower bound on the expected rank of a rational elliptic curve $E$ over an $S_n$-extension of $\mathbb{Q}$ for general $n$.
	
	\begin{theorem}[McKay conjecture for $p=2$, see \cite{gunter}]\label{mckay}
		Let $G$ be a finite group and let $P\leq G$ be a $2$-Sylow subgroup. For $H\leq G$, let $\textup{Irr}(H)$ denote the set of isomorphism classes of complex irreducible representations of $H$. Let $\textup{Irr}'(H)=\{\rho\in\textup{Irr}(H): 2\nmid \dim\rho\}$. Then 
		\begin{equation}
			\mid \textup{Irr}'(G) \mid = \mid \textup{Irr}'(N_G(P)) \mid
		\end{equation}
		where $N_G(P)$ is the normaliser of $P$ in $G$. 
	\end{theorem}
	
	\begin{lemma}[See Exercise 4.14 in \cite{fulton}]\label{sndim}
		For $n>6$, there are precisely four representations of $S_n$ of dimension less than $n$. Two of these have dimension $1$ and two have dimension $n-1$.
	\end{lemma}
	
	\begin{lemma}\label{abpk}
		Let $P_k$ be the $2$-Sylow subgroup of $S_{2^k}$. Then $P_k/[P_k,P_k]\cong C_2^k$, where $[P_k,P_k]$ denotes the commutator subgroup.
	\end{lemma}
	
	\begin{proof}
		We have $P_k=C_2^{\wr k}=C_2\wr\dots\wr C_2$, the wreath product with $k$ copies of $C_2$ (see, for example, \cite{dixon} Example 2.6.1), so we proceed by induction. For $k=1$ we have $P_1=C_2^{\wr 1}=C_2$. We assume that $P_k/[P_k,P_k]\cong C_2^k$. Now, if $H$ is the abelianisation of $G$ then the abelianisation of $G\wr C_2$ is the abelianisation of $H\wr C_2$. Hence, $P_{k+1}/[P_{k+1},P_{k+1}]$ is the abelianisation of $C_2^k\wr C_2$ which is $C_2^{k+1}$.
	\end{proof}

	\begin{theorem}\label{snthm}
		Let $E/\mathbb{Q}$ be any elliptic curve with $w(E/\mathbb{Q})=-1$ and let $K$ be any extension of $\mathbb{Q}$ with $\textup{Gal}(K/\mathbb{Q})=S_n$. Suppose that $(N_E,\Delta_K)=1$. If the parity conjecture holds then 
		\begin{equation}
			\textup{rk}(E/K)\geq \frac{n}{2}\left( \prod_k 2^{ka_k} - 2\right),
		\end{equation}
		where $n=\sum_k a_k 2^k$ is the binary expansion of $n$ with $a_k=0$ or $1$. 
	\end{theorem}

	\begin{proof}
		By Lemma \ref{permutation}, the parity conjecture for $E$ over all subfields of $K$ implies the parity conjecture for the twist of $E$ by any irreducible self-dual representation $\rho$ of $\text{Gal}(F/\mathbb{Q})$ (the fact that all representations of $S_n$ can be written as linear combination of permutation modules follows from \cite{james} \S6.1). Thus, it suffices to prove that the parity conjecture for twists implies the result. 
		
		The irreducible representations of $S_n$ are self-dual and there are two one-dimensional irreducible representations. Thus, by Theorem \ref{order2char}, to find a lower bound on the rank, it suffices to find a lower bound on the number of pairs of odd-dimensional irreducible representations of $S_n$. 
		
		Theorem \ref{mckay} says that the number of odd-dimensional irreducible representations of $S_n$ is equal to the number of one-dimensional representations of a $2$-Sylow subgroup of $S_n$, since the normaliser of a $2$-Sylow subgroup of $S_n$ is itself and the dimension of an irreducible representation divides the order of the group. Since $n=\sum_k a_k 2^k$, the $2$-Sylow subgroup of $S_n$ is the $2$-Sylow subgroup of $\prod_{k:a_k\neq 0} S_{2^k}$. By Lemma \ref{abpk}, the number of one-dimensional representations of the $2$-Sylow subgroup of $S_{2^k}$ is $2^k$. Hence, the number of one-dimensional representations of the $2$-Sylow subgroup of $\prod_{k:a_k\neq 0} S_{2^k}$ is $\prod_k 2^{ka_k}$. Theorem \ref{mckay} then tells us that there are $\frac{1}{2}\prod_k 2^{ka_k}$ pairs of odd-dimensional representations of $S_n$. 
		
		By Lemma \ref{sndim}, for $n>6$ other than the trivial and the sign representation, all of the odd-dimensional representations have dimension at least $n$, unless $n$ is even in which case there are also two $(n-1)$-dimensional representations. Hence, assuming the parity conjecture for twists, for $n>6$ we have
		\begin{equation}
			\text{rk}(E/K)=\text{dim}(E(K)\otimes_\mathbb{Z}\mathbb{C})\geq \frac{n}{2}\left( \prod_k 2^{ka_k} - 2\right). 
		\end{equation}
		Using Theorem \ref{order2char}$(ii)$, a simple check shows that the result is also true for $n\leq 6$.
	\end{proof}
	
	\begin{example}
		Fix $E:y^2+y=x^3-x$ and $K$ an $S_{14}$-extension of $\mathbb{Q}$ with $37\nmid\Delta_K$. Then $w(E/\mathbb{Q})=-1$ and $14=2+2^2+2^3$, so assuming the parity conjecture for twists, $\text{rk}(E/K)\geq \frac{14}{2}\cdot (2^{6}-2)=434$.
	\end{example}
	
	\begin{remark}
		The bound in Theorem \ref{snthm} is much lower than we expect the rank to be since most of the representations have dimension significantly larger than $n$. For example, let $K$ be an $S_7$-extension of $\mathbb{Q}$ and let $E/\mathbb{Q}$ with $w(E/\mathbb{Q})=-1$. Our bound says that $\text{rk}(E/K)\geq 21$. However, the dimensions of the odd-dimensional irreducible representations of $S_7$ are $1$, $1$, $15$, $15$, $21$, $21$, $35$ and $35$. So, assuming the parity conjecture for twists, Theorem \ref{order2char}$(ii)$ gives $\text{rk}(E/K)\geq 72$. 
	\end{remark}

	\subsection{Heegner hypothesis}
	When studying Heegner points, one often imposes the \textit{Heegner hypothesis}, that the quadratic field $K$ is imaginary and that all primes of bad reduction of $E/\mathbb{Q}$ split in $K$ (e.g. in Kolyvagin's work on the Birch--Swinnerton-Dyer conjecture, see \cite{kolyvagin}). From root number calculations, this ensures the odd order of vanishing of $L$-functions of $E$ over $K$ and of certain twists of $E$, as we now explain. 
	
	\begin{proposition}\label{heegner}
		Let $E/\mathbb{Q}$ be an elliptic curve and let $K$ be a quadratic extension of $\mathbb{Q}$ that satisfies the Heegner hypothesis. Let $F$ be a $C_n$-extension of $K$ so that $\textup{Gal}(F/\mathbb{Q})=D_{2n}$, the dihedral group with $2n$ elements. Assume that $(N_E,\Delta_F)=1$. Then
		
		\begin{enumerate}[(i)]
			\item For $\chi$ any of the $1$-dimensional representations of $\textup{Gal}(F/K)$, $\textup{ord}_{s=1}L(E/K,\chi,s)$ is odd.  
			\item  The order of vanishing of the $L$-function of $E/F$ satisfies $\textup{ord}_{s=1}L(E/F,s)\geq n$.
			\item Assuming the parity conjecture for twists, $\textup{rk}(E/F)\geq n$. 
		\end{enumerate}
	\end{proposition}
	
	\begin{proof}
		Let $H=\textup{Gal}(F/K)$ and $G=\textup{Gal}(F/\mathbb{Q})=D_{2n}$. The representations of $D_{2n}$ are all self-dual. The $1$-dimensional representations are $\mathds{1}$, the representation $\epsilon$ that factors through $\text{Gal}(K/\mathbb{Q})$, and, if $n$ is even, two more representations $\epsilon_1$ and $\epsilon_2$. There are $\frac{n-1}{2}$ irreducible $2$-dimensional representations if $n$ is odd and $\frac{n-2}{2}$ irreducible $2$-dimensional representations if $n$ is even. Call the irreducible $2$-dimensional representations $\rho_i$. 
		
		(i) By the inductivity of $L$-functions, for $\chi$ any of the $1$-dimensional representations of $\text{Gal}(F/K)$,
		\begin{equation}
			L(E/K,\chi,s)=L(E/\mathbb{Q},\textup{Ind}_H^G \chi,s)=L(E/\mathbb{Q},\rho,s),
		\end{equation}
		where $\rho$ is a $2$-dimensional representation of $D_{2n}$. The determinant character of $\rho$ is $\epsilon$ since $\rho$ is either $\mathds{1}\oplus\epsilon$, $\epsilon_1\oplus\epsilon_2$ or $\rho_i$, all of which have determinant character $\epsilon$. Hence, by the Heegner hypothesis on $K$ and in the notation of Theorem \ref{rootnostwists}, $\alpha_{\rho}$ is negative and satisfies $(\frac{\alpha_{\rho}}{N_E})=+1$. Thus,
		\begin{equation}
			w(E/\mathbb{Q},\rho)=(\pm 1)^2\cdot (-1)\cdot (+1)=-1. 
		\end{equation}
		By the inductivity of root numbers (Theorem \ref{artinformalism}$(iii)$), $w(E/K,\chi)=-1$ and, since we know the analytic continuation of $L(E/K,\chi,s)$ and that it satisfies a functional equation whose sign is $w(E/K,\chi)$ (see e.g \cite{shimura} Proposition 4·13), this implies that $\textup{ord}_{s=1}L(E/K,\chi,s)$ is odd.
		
		(ii) By the inductivity of $L$-functions,
		\begin{equation}
			L(E/F,s)=\prod_{\chi} L(E/K,\chi,s),
		\end{equation}
		where $\chi$ runs over the $1$-dimensional representations of $\text{Gal}(F/K)$. The result follows from part (i).
		
		(iii) Assuming the parity conjecture for twists, $\langle \rho,E(F)\otimes_\mathbb{Z}\mathbb{C} \rangle\geq 1$ where $\rho$ is $\rho_i$, $\mathds{1}\oplus\epsilon$ or $\epsilon_1\oplus\epsilon_2$ by the root number calculation in part (i). Hence $\text{rk}(E/F)\geq n$.
	\end{proof}

	\subsection{Artin representations that always appear with even multiplicity in $E(K)\otimes_\mathbb{Z}\mathbb{C}$}\label{evenartin}
	
	The parity conjecture for twists tells us the parity of the multiplicity of a representation $\rho$ in $E(K)\otimes_\mathbb{Z}\mathbb{C}$. Conversely, if $\rho$ is self-dual, irreducible and has Schur index $2$ (e.g. $\rho$ is symplectic) then it must appear with even multiplicity in $E(K)\otimes_\mathbb{Z}\mathbb{C}$. This has been checked to be compatible with root numbers in most cases, see \cite{rohrlich} Proposition E and \cite{sabitova}. On the other hand, root numbers suggest that certain representations must always appear with even multiplicity despite having Schur index $1$.
	
	\begin{proposition}[Rohrlich, \cite{rohrlich} Proposition D]\label{twistsphenomenon}
		Let $K$ be a Galois extension of $\mathbb{Q}$ where $\emph{\text{Gal}}(K/\mathbb{Q})= D_{2q}\times D_{2r}\times D_{2s}\times D_{2t}$ for distinct primes $q$, $r$, $s$, $t\geq 5$ and let $\tau$ be an irreducible $16$-dimensional representation of $\emph{\text{Gal}}(K/\mathbb{Q})$. Then $w(E/\mathbb{Q},\tau)=+1$ for every $E$ over $\mathbb{Q}$.
	\end{proposition}
	
	Proposition \ref{twistsphenomenon} tells us that, assuming the parity conjecture for twists, the multiplicity of $\tau$ in $E(K)\otimes_\mathbb{Z} \mathbb{C}$ is always even, which is as mysterious as the fact that every $E/\mathbb{Q}$ should have even rank over any number field which is Galois over $\mathbb{Q}$ and in which all places split into an even number of places (see Lemma \ref{even}). 
	
	\section{Minimalist conjecture for twists}\label{minconjtwists}
	There is a folklore `minimalist conjecture', which says that, generically, the rank of an elliptic curve $E/\mathbb{Q}$ is $0$ or $1$ depending on whether $w(E/\mathbb{Q})=+1$ or $-1$ (see e.g. \cite{timparity} Conjecture 8.2). We propose a minimalist conjecture for Artin twists, which gives the expected Galois module structure of $E(F)\otimes_\mathbb{Z}\mathbb{C}$ for most elliptic curves $E/\mathbb{Q}$, where $F$ is a finite Galois extension of $\mathbb{Q}$ (see Conjecture \ref{minconj} and Theorems \ref{minconjthm} and \ref{galmod} below). It says that, generically, if $E/\mathbb{Q}$ is an elliptic curve and $\rho$ is an irreducible self-dual Artin representation that factors through $F/\mathbb{Q}$, then $\rho$ appears in $E(F)\otimes_\mathbb{Z}\mathbb{C}$ with multiplicity $0$ or $1$, depending on whether the global root number of the twist of $E$ by $\rho$ is $+1$ or $-1$.

	\begin{example}\label{a5ex}
		Let $E/\mathbb{Q}$ be any elliptic curve with $w(E/\mathbb{Q})=-1$, and let $F$ be any extension of $\mathbb{Q}$ with $\text{Gal}(F/\mathbb{Q})=A_5$ and $(N_E,\Delta_F)=1$. By the conductor-discriminant formula, this means that $N_E$ is coprime to the conductors of the representations of $\text{Gal}(F/\mathbb{Q})$. We claim that $\text{rk}(E/F)\geq 12$. 
		
		There are five irreducible representations of $A_5$ and they are all self-dual: $\mathds{1}$, one $4$-dimensional representation $\rho$, one $5$-dimensional representation $\sigma$ and two $3$-dimensional representations $\tau_1$ and $\tau_2$. We can write
		\begin{equation}
			E(F)\otimes_\mathbb{Z}\mathbb{C}=\mathds{1}^{\oplus a}\oplus\rho^{\oplus b}\oplus\sigma^{\oplus c}\oplus\tau_1^{\oplus d}\oplus\tau_2^{\oplus e}.
		\end{equation}
		Since $\text{rk}(E/F)=\text{dim}(E(F)\otimes_\mathbb{Z}\mathbb{C})$, 
		\begin{equation}
			\text{rk}(E/F)=a+4b+5c+3d+3e. 
		\end{equation}
		We have $\text{det}(\sigma)=\text{det}(\tau_i)=\mathds{1}$ since $A_5$ has no subgroup of index $2$. Hence, by Theorem \ref{order2char}, the parity conjecture for twists implies that $a$, $c$, $d$ and $e$ are odd. In particular, $\text{rk}(E/F)\geq 12$. 
		
		We typically expect to have exactly
		\begin{equation}
			E(F)\otimes_\mathbb{Z}\mathbb{C}=\mathds{1}\oplus \sigma\oplus \tau_1\oplus\tau_2,
		\end{equation}
		and $\text{rk}(E/F)=12$. This is based on a `minimalist conjecture for twists'. The idea is that for any elliptic curve $E/\mathbb{Q}$ and any self-dual irreducible Artin representation $\tau$ that factors through an extension $F/\mathbb{Q}$, the contribution to the rank of $E/F$ by $\tau$ should be as small as the parity conjecture for twists will allow. 
	\end{example}

	\begin{notation}\label{ordering}
		For $E/\mathbb{Q}$ an elliptic curve, it is isomorphic to a unique curve $E_{A,B}:y^2=x^3+Ax+B$ where $A$, $B\in\mathbb{Z}$ and for all primes $p$, either $p^4\nmid A$ or $p^6\nmid B$. The naive height of $E_{A,B}$ is 
		\begin{equation}
			H(E_{A,B})=\max(4|A^3|,27B^2). 
		\end{equation}
		We say that $100\%$ of elliptic curves over $\mathbb{Q}$ satisfy a property $T$ if 
		\begin{equation}
			\lim_{X\to\infty}\frac{\#\{E_{A,B}|H(E_{A,B})<X\text{ and $E_{A,B}$ satisfies property $T$} \}}{\#\{E_{A,B}|H(E_{A,B})<X\}}=1,
		\end{equation}
		where for all primes $p$, either $p^4\nmid A$ or $p^6\nmid B$.
	\end{notation}
	
	\begin{conjecture}[Minimalist conjecture for twists]\label{minconj}
		Let $F$ be a Galois extension of $\mathbb{Q}$ and let $\rho$ be an irreducible Artin representation that factors through $F/\mathbb{Q}$. For $100\%$ of elliptic curves $E/\mathbb{Q}$,
		\begin{equation}
			\langle \rho,E(F)\otimes_\mathbb{Z}\mathbb{C} \rangle =\begin{cases} 
				0 &\text{ if } \rho \text{ is not self-dual,} \\
				0 &\text{ if } w(E/\mathbb{Q},\rho)=+1, \\
				1 &\text{ if } w(E/\mathbb{Q},\rho)=-1. 
			\end{cases} 
		\end{equation}
	\end{conjecture}
	
	\noindent Since the global root number of the twist is the conjectured sign in the functional equation relating $L(E/\mathbb{Q},\rho,s)$ and $L(E/\mathbb{Q},\rho^*,2-s)$, if $\rho$ is self-dual the parity of the order of vanishing of $L(E/\mathbb{Q},\rho,s)$ at $s=1$ is governed by $w(E/\mathbb{Q},\rho)$. For most elliptic curves, one expects the order of vanishing of the twisted $L$-function to be as small as is allowed by the functional equation. If $\rho$ is not self-dual, the functional equation does not tell us anything about the order of vanishing of $L(E/\mathbb{Q},\rho,s)$ at $s=1$. In this case, for most elliptic curves we expect the order of vanishing to be zero and that $\rho$ does not appear in $E(F)\otimes_\mathbb{Z}\mathbb{C}$.

	\begin{remark}
		If $\rho$ is a complex irreducible representation with Schur index $n$ that factors through $F/\mathbb{Q}$, the multiplicity of $\rho$ in $E(F)\otimes_\mathbb{Z}\mathbb{C}$ is divisible by $n$. The Schur index of an irreducible self-dual character can only be $1$ or $2$ by the Brauer--Speiser theorem \cite{fein, brauer,speiser}. So in the context of root numbers and the minimalist conjecture, Schur indices never force the multiplicity of an irreducible self-dual representation $\rho$ in $E(F)\otimes_\mathbb{Z}\mathbb{C}$ to be greater than $1$. 
	\end{remark}
	
	\begin{lemma}\label{coprimecond}
		Let $F$ be a Galois extension of $\mathbb{Q}$, let $\rho$ be an irreducible Artin representation that factors through $F/\mathbb{Q}$ and let $D$ be a non-zero integer. The minimalist conjecture for twists implies that for $100\%$ of elliptic curves $E/\mathbb{Q}$ with $(N_E,D)=1$,
		\vspace{-9pt}
		\begin{equation}
			\langle \rho,E(F)\otimes_\mathbb{Z}\mathbb{C} \rangle =\begin{cases} 
				0 &\text{ if } \rho \text{ is not self-dual,} \\
				0 &\text{ if } w(E/\mathbb{Q},\rho)=+1, \\
				1 &\text{ if } w(E/\mathbb{Q},\rho)=-1. 
			\end{cases} 
		\end{equation}
	\end{lemma}
	
	\begin{proof}
		It is known that, for any integer $D$, when ordered by the height of the coefficients there is a positive proportion of elliptic curves over $\mathbb{Q}$ whose conductor is coprime to $D$ (see e.g. Theorem 4.2(2) of \cite{cremona}; their ordering is slightly different, which may result in a different density, but we only need to know that this density is non-zero). In other words, $\lim_{X\to\infty}\frac{\#N}{\#T}=k>0$ where $T$ is the set of curves $E_{A,B}$ such that $H(E_{A,B})<X$ and $N\subset T$ is the set of elliptic curves $E_{A,B}$ such that $(N_E,D)=1$. Let $M\subset T$ be the set of curves $E_{A,B}$ that do not satisfy the conclusion of Conjecture \ref{minconj}, so that 
		\begin{equation}
			\frac{\#N\cap M}{\#N}\leq \frac{\#M}{\#N}=\frac{\#M}{\#T}\cdot \frac{\#T}{\#N}\rightarrow 0\cdot \frac{1}{k}=0 \quad \text{ as } \quad X\rightarrow \infty,
		\end{equation}
		where we have used Conjecture \ref{minconj} for $\#M/\#T\rightarrow 0$. This proves the lemma.
	\end{proof}
	
	\begin{lemma}\label{rootnodim}
		Let $E/\mathbb{Q}$ be an elliptic curve. Let $F$ be a Galois extension of $\mathbb{Q}$ that does not contain any quadratic number field. For any self-dual $\textup{Gal}(F/\mathbb{Q})$-representation $\rho$ whose conductor is coprime to $N_E$, 
		\begin{equation}
			w(E/\mathbb{Q},\rho)=w(E/\mathbb{Q})^{\textup{dim}\rho}. 
		\end{equation}
	\end{lemma}
	
	\begin{proof}
		Since $\text{Gal}(F/\mathbb{Q})$ has no index $2$ subgroup, in the notation of Theorem \ref{rootnostwists}, $\alpha_\rho=1$ for any self-dual $\text{Gal}(F/\mathbb{Q})$-representation $\rho$. So by Theorem \ref{rootnostwists} $w(E/\mathbb{Q},\rho)=w(E/\mathbb{Q})^{\text{dim}\rho}$.
	\end{proof}
	
	\begin{theorem}\label{minconjthm}
		Let $F$ be a Galois extension of $\mathbb{Q}$ that does not contain any quadratic number field. Write $W_G$ for the direct sum of the odd-dimensional irreducible self-dual representations of $G=\textup{Gal}(F/\mathbb{Q})$ and $k=\dim W_G$. Assuming the minimalist conjecture for twists, for $100\%$ of elliptic curves $E/\mathbb{Q}$ with $(N_E,\Delta_F)=1$
		\begin{equation}
			E(F)\otimes_\mathbb{Z}\mathbb{C}=\begin{cases}
				W_G &\text{ if } w(E/\mathbb{Q})=-1, \\
				0 & \text{ if }  w(E/\mathbb{Q})=+1.
			\end{cases} 
		\end{equation}
		In particular, for $100\%$ of elliptic curves $E/\mathbb{Q}$ with $(N_E,\Delta_F)=1$
		\begin{equation}
			\textup{rk}(E/F)=\begin{cases}
				k &\text{ if $w(E/\mathbb{Q})=-1$}, \\
				0 &\text{ if $w(E/\mathbb{Q})=+1$}.
			\end{cases}
		\end{equation} 
	\end{theorem}
	
	\begin{proof}
		Let $\rho$ be a representation of $\text{Gal}(F/\mathbb{Q})$. By the conductor-discriminant formula, this ensures that the conductor of $\rho$ is coprime to $N_E$. If $\rho$ is not self-dual, by Lemma \ref{coprimecond} the minimalist conjecture for twists implies that $\rho$ does not appear in $E(F)\otimes_\mathbb{Z}\mathbb{C}$ for $100\%$ of elliptic curves $E/\mathbb{Q}$ with $(N_E,\Delta_F)=1$. If $\rho$ is self-dual, since $\text{Gal}(F/\mathbb{Q})$ has no index $2$ subgroup, by Lemma \ref{rootnodim} $w(E/\mathbb{Q},\rho)=w(E/\mathbb{Q})^{\text{dim}\rho}$ for $E/\mathbb{Q}$ any elliptic curve with $(N_E,\Delta_F)=1$. By Lemma \ref{coprimecond}, the minimalist conjecture for twists implies that, for $100\%$ of elliptic curves $E/\mathbb{Q}$ with $(N_E,\Delta_F)=1$,
		\begin{equation}
			\langle \rho,E(F)\otimes_\mathbb{Z}\mathbb{C} \rangle =\begin{cases} 
				1 &\text{ if $\rho$ has odd dimension and $w(E/\mathbb{Q},\rho)=-1$,} \\
				0 &\text{ otherwise},
			\end{cases}
		\end{equation}
		which gives us the result.
	\end{proof}
	
	\begin{remark}
		Theorem \ref{minconjthm} says that elliptic curves of rank $1$ over $\mathbb{Q}$ tend to gain many points over Galois extensions with no quadratic subfields, while elliptic curves of rank $0$ over $\mathbb{Q}$ tend not to obtain any new points. It suggests that the equation $y^2+y=x^3-x$ is `easy to solve' whereas the equation $y^2+y=x^3-x^2$ is `hard to solve' over most Galois extensions of $\mathbb{Q}$ with no quadratic subfield.
	\end{remark}
	
	\begin{remark}
		In the setting of Theorem \ref{minconjthm}, we expect even-dimensional irreducible Artin representations to be very rare inside Mordell--Weil groups.
	\end{remark}
	
	\begin{theorem}\label{galmod}
		Let $F$ be a Galois extension of $\mathbb{Q}$. Assuming the minimalist conjecture for twists, for $100\%$ of $E/\mathbb{Q}$ with $(N_E,2\Delta_F)=1$,  the Galois module $E(F)\otimes_\mathbb{Z}\mathbb{C}$ depends only on $N_E\pmod{8\Delta_F}$ and $w(E/\mathbb{Q})$.
	\end{theorem}
	
	\begin{proof}
		By Lemma \ref{coprimecond}, it is enough to show that $w(E/\mathbb{Q},\rho)$ only depends on $N_E\pmod{8\Delta_F}$ and $w(E/\mathbb{Q})$ for every self-dual representation $\rho$ of $\text{Gal}(F/\mathbb{Q})$. Let $\rho$ be any self-dual representation of $\text{Gal}(F/\mathbb{Q})$ and suppose we have $E$ and $E'$ elliptic curves over $\mathbb{Q}$ with $N_E\equiv N_{E'}\pmod{8\Delta_F}$, $(N_E,2\Delta_F)=1=(N_{E'},2\Delta_F)$ and $w(E/\mathbb{Q})=w(E'/\mathbb{Q})$. By the conductor-discriminant formula this ensures that $N_E$ and $N_{E'}$ are coprime to the conductor of $\rho$, so we can apply Theorem \ref{rootnostwists} to calculate $w(E/\mathbb{Q},\rho)$. In the notation of Theorem~\ref{rootnostwists}, $\alpha_\rho\mid\Delta_F$. Since $N_E\equiv N_{E'}\pmod{8\alpha_\rho}$ and $N_E$ and $N_{E'}$ are odd, by quadratic reciprocity we have $ \left(\frac{\alpha_\rho}{N_E}\right)=\left(\frac{\alpha_\rho}{N_{E'}}\right)$. Hence, by Theorem \ref{rootnostwists}, $w(E/\mathbb{Q},\rho)=w(E'/\mathbb{Q},\rho)$.
	\end{proof}
	
	\begin{example}
		Let $F$ be the splitting field of the polynomial 
		\begin{equation}
			x^{10} + 5x^8 + 15x^6 + 20x^4 + 25x^2 + 15.
		\end{equation}
		Then $\text{Gal}(F/\mathbb{Q})=D_{10}$ and $\Delta_F=-3^{5}\cdot 5^{13}$. Theorem \ref{galmod} tells us that, for $100\%$ of elliptic curves $E/\mathbb{Q}$ with $(N_E,2\Delta_F)=1$, the Galois module structure of $E(F)\otimes_\mathbb{Z}\mathbb{Q}$ depends on $N_E$ modulo $8\cdot 3^{5}\cdot 5^{13}$ and $w(E/\mathbb{Q})$. In fact, we will show that it only depends on $N_E\pmod{15}$ and $w(E/\mathbb{Q})$.
		
		The representations of $\textup{Gal}(F/\mathbb{Q})$ are the trivial representation $\mathds{1}$, one $1$-dimensional representation $\epsilon$ that factors through $\mathbb{Q}(\sqrt{-15})$ and two $2$-dimensional representations $\rho_1$ and $\rho_2$, both of which have determinant~$\epsilon$. Thus when $(N_E,2\Delta_F)=1$, by Theorem \ref{rootnostwists}, 
		\begin{equation}
			w(E/\mathbb{Q},\epsilon)=\begin{cases}
				-w(E/\mathbb{Q}) &\text{ if } \left(\frac{-15}{N_E}\right)=+1, \\
				w(E/\mathbb{Q}) &\text{ if } \left(\frac{-15}{N_E}\right)=-1,
			\end{cases} \hspace{30pt} w(E/\mathbb{Q},\rho_i)=\begin{cases}
				-1 &\text{ if } \left(\frac{-15}{N_E}\right)=+1, \\
				+1 &\text{ if } \left(\frac{-15}{N_E}\right)=-1. 
			\end{cases}
		\end{equation}
		We have 
		\begin{equation}
			\left(\frac{-15}{N_E}\right)=\begin{cases}
				+1 &\text{ if } N_E\equiv 1, \ 2, \ 4 \text{ or } 8 \pmod{15}; \\
				-1 &\text{ if } N_E\equiv 7, \ 11, \ 13 \text{ or } 14 \pmod{15}.
			\end{cases}
		\end{equation}
		Note that when $(N_E,2\Delta_F)=1$ these congruence classes are the only possibilities for $N_E$ modulo $15$. Hence, assuming the minimalist conjecture for twists, for $100\%$ of elliptic curves $E/\mathbb{Q}$ with $(N_E,2\Delta_F)=1$ 
		
		\begin{equation}
			E(F)\otimes_\mathbb{Z}\mathbb{C}=\begin{cases}
				\epsilon\oplus\rho_1\oplus\rho_2 &\text{ if }N_E\equiv 1, \ 2, \ 4 \text{ or } 8 \pmod{15} \text{ and } w(E/\mathbb{Q})=+1; \\
				\mathds{1}\oplus\rho_1\oplus\rho_2 &\text{ if } N_E\equiv 1, \ 2, \ 4 \text{ or } 8 \pmod{15} \text{ and } w(E/\mathbb{Q})=-1; \\
				0 &\text{ if } N_E\equiv 7, \ 11, \ 13 \text{ or } 14 \pmod{15} \text{ and } w(E/\mathbb{Q})=+1; \\
				\mathds{1}\oplus\epsilon&\text{ if } N_E\equiv 7, \ 11, \ 13 \text{ or } 14 \pmod{15} \text{ and } w(E/\mathbb{Q})=-1.
			\end{cases}
		\end{equation}
	\end{example}

	\appendix
	
	\section{Local root numbers of elliptic curves over $\mathbb{Q}_2$}\label{rootnosat2table}
	
	The following table can be used to calculate $w(E/\mathbb{Q}_2)$, including when $E/\mathbb{Q}_2$ has additive reduction. It is based on Table III in \cite{rizzo}, which generalises results of Halberstadt \cite{halberstadt} to non-minimal Weierstrass equations. For a discussion on what can be done to calculate root numbers over extensions of $\mathbb{Q}_2$, see \S\ref{localrootnos}.
	
	\begin{notation}\label{cprime}
		Let $E/\mathbb{Q}_2$ be an elliptic curve given by a Weierstrass equation.
		\begin{itemize}
			\item Let $m=\text{min}\{\lfloor v_2(\Delta_E)/12 \rfloor, \lfloor v_2(c_6)/6 \rfloor, \lfloor v_2(c_4)/4 \rfloor\}$. Define
			\begin{equation}
				(C_{\Delta},C_6,C_4)=(v_2(\Delta_E)-12m,v_2(c_6)-6m,v_2(c_4)-4m). 
			\end{equation}
			In other words, take $(v_2(\Delta_E),v_2(c_6),v_2(c_4))$ and subtract multiples of $(12,6,4)$ entry-wise until reaching the smallest triple of non-negative integers $(C_\Delta,C_6,C_4)$. By convention, $v_2(0)=\infty$.
			\item $x'=x/2^{v_2(x)}$.
			\item $c_{6,e}=c_6/2^{v_2(c_6)-C_6+e}$; if $c_6=0$ then $c_{6,e}=0$. 
		\end{itemize}
	\end{notation}
	
	\begin{example}
		If $E:y^2=x^3-64x-128$ then $\Delta_E=2^{18}\cdot 37$, $c_6=2^{12}\cdot 3$ and $c_4=2^{10}\cdot 3$ so $(C_\Delta,C_6,C_4)=(6,6,6)$. We have $c_6'=3$ and the table below tells us that $w(E/\mathbb{Q}_2)=-1$. 
	\end{example}

	\begin{center}
		
		{\setlength{\LTleft}{-0.25in}
			\setlength{\LTright}{0in}
			\begin{longtable}{r r r | l | c }
				$C_\Delta$  & $C_6$ &  $C_4$ &   Extra conditions & $w(E/\mathbb{Q}_2)$  \\ 
				\hline\hline
				$0$ & $0$ & $0$  & $c_6'\equiv 3\pmod{4}$ & $+1$\\ 
				\hline
				$\geq 0$  & $0$ & $0$ & $c_6'\equiv 1\pmod{4}$ & $-1$\\
				\hline 
				$0$ & $3$ & $3$ & $c_4'\equiv 1\pmod{4}$ and $c_6'\equiv \pm 1\pmod{8}$ or $c_4'\equiv 3 \pmod{4}$ and $c_6'\equiv 1 \text{ or } 3 \pmod{8}$ & $+1$ \\
				\hline
				$0$& $3$& $3$ & $c_4'\equiv 1\pmod{4}$ and $c_6'\equiv 3 \text{ or } 5 \pmod{8}$ or $c_4'\equiv 3 \pmod{4}$ and $c_6'\equiv 5 \text{ or } 7 \pmod{8}$ & $-1$ \\
				\hline
				$0$ & $3$ & $\geq 4$ & $c_6'\equiv 1\pmod{4}$ & $+1$\\
				\hline
				$0$ & $3$ & $\geq 4$ & $c_6'\equiv 3\pmod{4}$ & $-1$ \\  
				\hline
				$0$ & $4$ & $2$ & $c_4'\equiv 1 \pmod{4}$ and $c_4'+4c_6'\equiv 9 \text{ or } 13\pmod{16}$ & $+1$ \\  
				\hline 
				$0$ & $4$ & $2$ & $c_4'\equiv 1 \pmod{4}$ and $c_4'+4c_6'\not \equiv 9 \text{ or } 13\pmod{16}$ & $-1$ \\  
				\hline
				$0$ & $\geq 4$ & $2$ &  $c_4'\equiv 3 \pmod{4} $ and $C_6=4$&
				$+1$ \\ 
				\hline 
				$0$ & $\geq 4$ & $2$ &  $c_4'\equiv 3 \pmod{4} $ and $C_6\neq 4$ & $-1$\\  
				\hline
				$0$ & $5$ & $2$ & $c_4'\equiv 1\pmod{4}$ and $c_4'+4c_{6,4}\equiv 5 \text{ or } 9 \pmod{16}$ & $-1$\\   
				\hline
				$0$ & $5$ & $2$ & $c_4'\equiv 1\pmod{4}$ and $c_4'+4c_{6,4}\not\equiv 5 \text{ or } 9 \pmod{16}$  & $+1$\\ 
				\hline
				$0$ & $\geq 6$ & $2$ & $c_4'\equiv 1\pmod{4}$ and $c_4'+4c_{6,4}\equiv 5 \text{ or } 9 \pmod{16}$ & $+1$\\   
				\hline
				$0$ & $\geq 6$ & $2$ & $c_4'\equiv 1\pmod{4}$ and $c_4'+4c_{6,4}\not\equiv 5 \text{ or } 9 \pmod{16}$  & $-1$\\ 
				\hline
				$\geq 1$ & $0$ & $0$ &  $c_6'\equiv 3\pmod{4}$ and $c_6'\equiv 3\pmod{8}$ & $+1$\\  
				\hline
				$\geq 1$ &  $0$ & $0$ & $c_6'\equiv 3\pmod{4}$ and $c_6'\not\equiv 3\pmod{8}$ & $-1$\\  
				\hline
				$1$ & $3$ & $2$ & $c_4'+4c_6'\equiv 3 \pmod{16}$ or $c_4'\equiv 11\pmod{16}$ & $+1$\\  
				\hline
				$1$ & $3$ & $2$ & $c_4'+4c_6'\not\equiv 3 \pmod{16}$ and $c_4'\not\equiv 11\pmod{16}$ & $-1$\\ 
				\hline
				$2$ & $3$ & $2$ & $\Delta_E'\equiv c_6'\pmod{4}$ & 
				$+1$\\  
				\hline
				$2$ & $3$ & $2$ & $\Delta_E'\not\equiv c_6'\pmod{4}$ & 
				$-1$\\  
				\hline
				$2$ & $4$ & $3$ & $c_4'+c_6'\equiv 0 \text{ or } 6\pmod{8}$  &
				$+1$ \\ 
				\hline 	
				$2$ & $4$ & $3$ & $c_4'+c_6'\not\equiv 0 \text{ or } 6\pmod{8}$  &
				$-1$ \\ 
				\hline
				$2$ & $4$ & $\geq 4$ & $c_6'\equiv 1\pmod{4}$ & $+1$ \\  
				\hline
				$2$  & $4$ & $\geq 4$ & $c_6'\not\equiv 1\pmod{4}$ & $-1$ \\  
				\hline
				$3$ & $3$ & $2$  & $\Delta_E'\equiv 3\pmod{4}$ & 
				$+1$ \\  
				\hline
				$3$ & $3$ & $2$ & $\Delta_E'\not\equiv 3\pmod{4}$ & 
				$-1$ \\  
				\hline
				$3$ & $5$ & $3$ & $2c_6'+c_4'\equiv 1\text{ or } 3\pmod{8}$ & $+1$ \\  
				\hline
				$3$ & $5$ & $3$ & $2c_6'+c_4'\equiv 5\text{ or } 7\pmod{8}$ & $-1$ \\ 
				\hline
				$3$ & $\geq 6$ & $3$ & $c_4'\equiv 5\text{ or } 7 \pmod{8}$ &
				$+1$ \\  
				\hline
				$3$ & $\geq 6$ & $3$ & $c_4'\equiv 1\text{ or } 3 \pmod{8}$ &
				$-1$ \\  
				\hline
				$\geq 4$ & $3$ & $2$ & $c_6'\equiv 3\pmod{4}$ & $+1$ \\  
				\hline
				$\geq 4$ & $3$ & $2$ & $c_6'\not\equiv 3\pmod{4}$ & $-1$ \\  
				\hline
				$4$ & $5$ & $4$ & $c_4'\equiv c_6'\pmod{4}$ and $c_4'\equiv 1\pmod{4}$ &
				$+1$ \\  
				\hline
				$4$ & $5$ & $4$ & $c_4'\equiv c_6'\pmod{4}$ and $c_4'\equiv 3\pmod{4}$ &
				$-1$ \\ 
				\hline
				$4$ & $5$ & $4$ & $c_4'\equiv 1\equiv -c_6'\pmod{4}$ and $c_4'c_6'\equiv 3\pmod{8}$ & 
				$+1$ \\  
				\hline
				$4$ & $5$ & $4$ & $c_4'\equiv 1\equiv -c_6'\pmod{4}$ and $c_4'c_6'\not\equiv 3\pmod{8}$ & 
				$-1$ \\
				\hline 
				$4$ & $5$ & $4$ & $c_6'\equiv 1\equiv -c_4'\pmod{4}$ & $-1$\\  
				
				\hline
				$4$ & $5$ & $5$ & $c_6'\equiv 1\pmod{4}$ and  $c_6'\equiv 5\pmod{8}$ &
				$+1$\\  
				\hline
				$4$ & $5$ & $5$ & $c_6'\equiv 1\pmod{4}$ and  $c_6'\equiv 1\pmod{8}$ &
				$-1$\\  
				\hline
				$4$ & $5$ & $\geq 5$ & $c_6'\equiv 3\pmod{4}$ and $C_4=5$ &
				$+1$\\  
				\hline
				$4$  & $5$ & $\geq 5$ & $c_6'\equiv 3\pmod{4}$ and $C_4\neq 5$ &
				$-1$\\ 
				\hline
				$4$ & $5$ & $\geq 6$ & $c_6'\equiv 1\pmod{4}$  &
				$-1$ \\  
				\hline
				$6$ & $6$ & $5$ & $c_4'\equiv 3\pmod{4}$ & $+1$\\ 
				\hline
				$6$ & $6$ & $5$ & $c_4'\equiv 1\pmod{4}$ & $-1$\\ 
				\hline
				$6$ & $6$ & $\geq 6$ & $c_6'\equiv 1\pmod{4}$; & $+1$\\
				\hline
				$6$ & $6$ & $\geq 6$ & $c_6'\equiv 3\pmod{4}$; & $-1$\\
				\hline
				$6$ & $\geq 7$ & $4$ & $c_4'\equiv 1\pmod{4}$ and $C_6=7$ & 
				$+1$\\
				\hline
				$6$ & $\geq 7$ & $4$ & $c_4'\equiv 1\pmod{4}$ and $C_6\neq 7$ & 
				$-1$\\
				\hline
				$6$ & $\geq 7$ & $4$  & $c_4'\equiv 3\pmod{4}$ and $c_4'-4c_{6,7}\equiv 7\text{ or }11\pmod{16}$ & $+1$\\
				\hline
				$6$ & $\geq 7$ & $4$ & $c_4'\equiv 3\pmod{4}$ and $c_4'-4c_{6,7}\not\equiv 7\text{ or }11\pmod{16}$ & $-1$ \\
				\hline 
				$7$ & $6$ & $4$  & $c_6'\equiv 5\text{ or } 5c_4'\pmod{8}$ &
				$+1$ \\  
				\hline 
				$7$ & $6$ & $4$   & $c_6'\not\equiv 5\text{ or } 5c_4'\pmod{8}$ &
				$-1$ \\ 
				\hline
				$8$ & $6$ & $4$ & $2c_6'+c_4'\equiv 3 \pmod{16}$ &
				$+1$ \\
				\hline
				$8$ & $6$ & $4$  & $2c_6'+c_4'\equiv 15 \pmod{16}$ &
				$-1$ \\
				\hline
				$8$ & $6$ & $4$ & $2c_6'+c_4'\equiv 23 \pmod{32}$  & $+1$ \\  
				\hline
				$8$ & $6$ & $4$  & $2c_6'+c_4'\equiv 7 \pmod{32}$  & $-1$ \\  
				\hline
				$8$ & $6$ & $4$ & $2c_6'+c_4'\equiv 11 \pmod{16} $ & 
				$-1$ \\  
				\hline
				$8$ & $7$ & $5$  & $2c_4'+c_6'\equiv 7 \pmod{8}$ or $c_6'\equiv 3\pmod{8}$ & $+1$\\  
				\hline
				$8$ & $7$ & $5$  & $2c_4'+c_6'\not\equiv 7 \pmod{8}$ and $c_6'\not\equiv 3\pmod{8}$ & $-1$\\ 
				\hline
				$8$ & $7$ & $6$ & $c_6'\equiv 1\pmod{4}$ and $2c_4'+c_6'\equiv 3 \pmod{8}$ &
				$+1$ \\  
				\hline
				$8$ & $7$ & $6$ & $c_6'\equiv 1\pmod{4}$ and $2c_4'+c_6'\not\equiv 3 \pmod{8}$ &
				$-1$ \\  
				\hline
				$8$ & $7$ & $\geq 6$ & $c_6'\equiv 3\pmod{4}$ and $C_4=6$ & $+1$\\  
				\hline
				$8$ & $7$ & $\geq 6 $ & $c_6'\equiv 3\pmod{4}$ and $C_4\neq 6$ & $-1$\\  
				\hline 
				$8$ & $7$ & $\geq 7$ & $c_6'\equiv 1\pmod{4}$ & 
				$-1$ \\  
				\hline
				$9$ & $6$ & $4$  & $2c_6'+c_4'\equiv 11 \pmod{32}$ or $c_6'\equiv 7\pmod{8}$ &  $+1$\\  
				\hline
				$9$ & $6$ & $4$  & $2c_6'+c_4'\not\equiv 11 \pmod{32}$ and $c_6'\not\equiv 7\pmod{8}$ &  $-1$ \\  
				\hline
				$9$ & $8$ & $5$ & $2c_6'+c_4'\equiv \pm 1 \pmod{8}$ & 
				$+1$ \\  
				\hline
				$9$ & $8$ & $5$  & $2c_6'+c_4'\equiv 3\text{ or } 5 \pmod{8}$ & 
				$-1$ \\  
				\hline
				$9$ & $\geq 9$ & $5$  & $c_4'\equiv 1\text{ or } 3\pmod{8}$ & $+1$ \\  
				\hline
				$9$ & $\geq 9$ & $5$  & $c_4'\equiv 5\text{ or } 7\pmod{8}$ & $-1$ \\ 
				\hline
				$10$ & $6$ & $4$ & $c_6'\equiv 1\pmod{4}$  & 
				$+1$\\  
				\hline
				$10$  & $6$ & $4$ & $c_6'\equiv 3\pmod{4}$ and $c_4'-2c_6'\equiv 3\text{ or } 19 \pmod{64}$ & $+1$\\  
				\hline
				$10$ & $6$ & $4$ & $c_6'\equiv 3\pmod{4}$ and $c_4'-2c_6'\not\equiv 3\text{ or } 19 \pmod{64}$ & $-1$\\ 
				\hline
				$10$ & $8$ & $6$  & $c_4'c_6'\equiv 3\pmod{4}$ & 
				$+1$ \\ 
				\hline
				$10$ & $8$ & $6$  & $c_4'c_6'\equiv 1\pmod{4}$ & 
				$-1$ \\  
				\hline
				$10$ & $8$ & $\geq 7$ & $c_6'\equiv 1\pmod{4}$ &
				$+1$ \\  
				\hline
				$10$  & $8$ & $\geq 7$  & $c_6'\equiv 3\pmod{4}$ &
				$-1$ \\  
				\hline
				$11$ & $6$ & $4$  & $c_6'\equiv 1\pmod{4}$ &
				$+1$ \\  
				\hline
				$11$ & $6$ & $4$  & $c_6'\equiv 3\pmod{4}$ and $c_6'\equiv 3\pmod{8}$ & $+1$\\  
				\hline
				$11$ & $6$ & $4$ & $c_6'\equiv 3\pmod{4}$ and $c_6'\equiv 7 \pmod{8}$ & $-1$ \\  
				\hline
				
			\end{longtable}
		}
	\end{center}
	\vspace{-18pt}
	There appear to be a couple of typos in Table III of \cite{rizzo}: it shows $(>0,0,0)$ instead of $(\geq 0,0,0)$ and the definition of $c_{6,e}$ is not consistent with Halberstadt's definition. We have checked that all possible cases for $(C_\Delta,C_6,C_4)$ are covered by the table, using the equation $1728\Delta_E=c_4^3-c_6^2$. We have also numerically checked that the table is consistent with the implementation in Magma \cite{magma} of Halberstadt's table in \cite{halberstadt} for over one million curves (including examples from each entry).

\end{document}